\documentclass[onefignum,onetabnum]{siamart190516}
\usepackage[utf8]{inputenc}

\usepackage{amsmath}
\usepackage{amsfonts}
\usepackage{amssymb}
\usepackage{fixmath}
\usepackage{subcaption,booktabs}
\usepackage{microtype}
\usepackage{tikz}
\usetikzlibrary{arrows.meta,calc,decorations.markings,math,arrows.meta}
\usepackage{ntheorem}

\usepackage[normalem]{ulem}

\newcommand\x{1.4}
\newcommand\y{1.5}

\newcommand\mX{\mathbf{X}}
\newcommand\mY{\mathbf{Y}}
\newcommand\mZ{\mathbf{Z}}
\newcommand\ma{\mathbf{A}}

\newcommand\mA{\mathbf{A}}
\newcommand\mB{\mathbf{B}}
\newcommand\mR{\mathbf{R}}
\newcommand\mT{\mathbf{T}}
\newcommand\myeq{\stackrel{\mathrm{def}}{=}}

\newcommand\mG{\mathbf{G}}
\newcommand\mgrad{\mathbf{J}}

\newcommand\mE{\mathbf{E}}
\newcommand\mF{\mathbf{F}}

\newcommand\mU{\mathbf{U}}
\newcommand\mS{\mathbf{S}}
\newcommand\mV{\mathbf{V}}
\newcommand\mI{\mathbf{I}}

\newcommand\mdu{\dot{\mathbf{U}}}

\newcommand\mdv{\dot{\mathbf{V}}}

\newtheorem{statement}{Proposition}[section]
\newcommand\compl{\mathcal{O}}
\newcommand\rank{r}

\makeatletter
\newcommand{\subalign}[1]{%
  \vcenter{%
    \Let@ \restore@math@cr \default@tag
    \baselineskip\fontdimen10 \scriptfont\tw@
    \advance\baselineskip\fontdimen12 \scriptfont\tw@
    \lineskip\thr@@\fontdimen8 \scriptfont\thr@@
    \lineskiplimit\lineskip
    \ialign{\hfil$\m@th\scriptstyle##$&$\m@th\scriptstyle{}##$\crcr
      #1\crcr
    }%
  }
}
\makeatother

\usepackage{url}
\usepackage{mathtools}

\usepackage[a4paper, total={6in, 9in}]{geometry}

\renewcommand{\vec}[1]{\mathbf{#1}}
\newcommand{\mat}[1]{\mathbf{#1}}
\newcommand{\tens}[1]{\mathcal{\textbf{#1}}} \newcommand{\tensel}[1]{\mathcal{#1}}

\newcommand{\reshape}{\texttt{reshape}}

\usepackage{algorithm}
\usepackage[noend]{algpseudocode}
\usepackage{xcolor}
\algrenewcommand\textproc{}%

\usepackage[numbers]{natbib}
\title{Automatic differentiation for Riemannian optimization on low-rank matrix and tensor-train manifolds}

\begin{document}

\author{Alexander Novikov\footnotemark[1] \footnotemark[4] \and Maxim Rakhuba\footnotemark[4] \and Ivan Oseledets\footnotemark[3] \footnotemark[1]}

\renewcommand{\thefootnote}{\fnsymbol{footnote}}

\footnotetext[1]{Marchuk Institute of Numerical Mathematics of the Russian Academy of Sciences, 119333 Moscow, Russia \email{sasha.v.novikov@gmail.com}.}
\footnotetext[3]{Skolkovo Institute of Science and Technology, Skolkovo Innovation Center, 121205 Moscow, Russia \email{i.oseldets@skoltech.ru}.}
\footnotetext[4]{HSE University, Pokrovsky Boulevard 11, Moscow, 109028 Russian Federation}

\renewcommand{\thefootnote}{\arabic{footnote}}

\date{}

\maketitle

\begin{abstract}
  In scientific computing and machine learning applications, matrices and more general multidimensional arrays (tensors) can often be approximated with the help of low-rank decompositions. 
  Since matrices and tensors of fixed rank form smooth Riemannian manifolds, one of the popular tools for finding low-rank approximations is to use Riemannian optimization.
  Nevertheless, efficient implementation  of Riemannian gradients and Hessians, required in Riemannian optimization algorithms, can be a nontrivial task in practice.
  Moreover, in some cases, analytic formulas are not even available. %
  In this paper, we build upon automatic differentiation and propose a method that, given an implementation of the function to be minimized, efficiently computes Riemannian gradients and matrix-by-vector products between an approximate Riemannian Hessian and a given vector.
\end{abstract}

\section{Introduction}

Automatic differentiation (AD) is a powerful tool for numerically calculating derivatives of functions specified as computer programs.
It significantly simplifies the programming of derivatives of complicated functions without loss of efficiency, providing better stability properties compared with classical numerical differentiation using finite differences.
AD is commonly used in applied mathematics, and in particular, it is at the core of deep learning success, allowing researchers to combine ever more complex neural networks from known modules, and train them without worrying about efficient gradient computation.

In this paper, we are concerned with applying AD to the minimization problem
\[
	\min_{\mX\in \mathcal{M}} f(\mX),
\] 
where $f\colon \mathbb{R}^{n_1\times \dots \times n_d} \to \mathbb{R}$ is a smooth function and $\mathcal{M}$ is a subset of $\mathbb{R}^{n_1\times \dots \times n_d}$ of fixed-rank matrices ($d=2$) or fixed-rank tensor-trains ($d>2$)~\cite{holtz-manifolds-fixed-rank-2011}. 
It is known that in both cases, $\mathcal{M}$ forms a Riemannian manifold. 
One can, therefore, apply Riemannian optimization algorithms~\cite{absil} that
are currently actively used for the development of state-of-the-art algorithms in numerical mathematics, partial differential equations and machine learning.
A realization of such algorithms requires specific knowledge of computational aspects of low-rank objects and is especially complicated for tensor decompositions, where a number of tricks have to be done to reduce rank dependence and ensure the stability of an algorithm. The AD technique proposed in this work allows for a significant simplification of this process.

We are concerned with computing Riemannian gradients and matrix-vector products with approximate Riemannian Hessians, which are the building blocks for Riemannian optimization algorithms.
Note that in the case of low-rank matrix or tensor-train manifolds, the matrix-vector product with the Hessian can be numerically unstable. It happens due to the presence of terms with inverted singular values~\cite{absil2013extrinsic}. We, therefore, consider multiplication by the \emph{approximate} Hessian with an omitted curvature term~\cite{kressner2016preconditioned,vandereycken2010riemannian,ro-jd-2018} (see details in Sec.~\ref{sec:riemannian-opt-briefer}).

In the proposed method, calculating the Riemannian gradient or matrix-vector product with the approximate Riemannian Hessian of a function has the same asymptotic complexity as evaluation of the function itself at a single point\footnote{This holds under the assumption that the function evaluation is at least as expensive as the cost of the orthogonalization operation, which is a necessary step in any Riemannian gradient computation. This assumption holds true for most practical functions (see Propositions~\ref{thm:riemannian-grad-complexity} and~\ref{thm:riemannian-hess-complexity} for more details).}.
Moreover, thanks to the implementation in TensorFlow (a Python library with AD support), the algorithms can be run both on CPUs and GPUs.

We numerically evaluate the performance of the proposed algorithms on several functions arising from solving systems of linear equations, the eigenvalue problem, the tensor completion problem and in the training of a machine learning model.

Our main contributions are:
\begin{itemize}
    \item We develop automatic differentiation algorithms for computing the Riemannian gradient and a matrix-vector product with the approximate Riemannian Hessian of a function for low-rank matrices and TT-tensors. Under mild assumptions, the asymptotic complexity of the proposed method equals the complexity of evaluating the function at one point.
    \item We implement the proposed algorithms in TensorFlow and make them available in $\texttt{T3F}$\footnote{\url{https://github.com/Bihaqo/t3f}} -- an open-source Python library for working with TT decomposition.
\end{itemize}

\paragraph{Related work}
There is a large body of work on creating libraries for working with tensors and tensor decompositions, which often include automatic differentiation abilities (see, e.g., \cite{tensorflow2015-whitepaper,kossaifi2016tensorly, tt-toolbox,ma2020autohoot,suess2017mpnum}, \texttt{tntorch}\footnote{\url{https://tntorch.readthedocs.io}}), but most of these libraries do not target the \emph{Riemannian automatic differentiation}, which is the focus of this paper. 
Typically, researchers compute the Riemannian gradients manually, but the Riemannian automatic differentiation libraries~\citep{townsend2016pymanopt,sommer2016automatic,koppel2018manifold} are gaining traction, empowering the Riemannian optimization community. However, existing Riemannian AD libraries lack low-rank tensor support. For low-rank matrices, \texttt{PyManOpt}~\citep{townsend2016pymanopt} supports Riemannian gradients, but no library supports multiplying the Riemannian Hessian by a given vector, which is required for second-order methods. 

Note that in \cite{townsend2016pymanopt}, an algorithm to compute the Riemannian gradient for low-rank matrices has already been proposed and implemented. 
Nevertheless, in this work, we present an alternative way of doing it avoiding inversions of singular values, which can be close to machine epsilon if the rank is overestimated.

A method for automatic second-order Riemannian differentiation for the manifold of low-rank tensors was proposed in~\cite{psenka2020second}. The authors focus on the curvature term of the Riemannian Hessian (which we omit as explained in Sec.~\ref{sec:riemannian-opt-briefer}) and assume that the other terms can be computed efficiently by a two-step procedure: first computing the Euclidean gradient or Hessian-by-vector product and then projecting it onto the tangent space. This is indeed efficient for some functions, but can be significantly slower than the approach proposed in this paper for some other function. Thus, the two papers complement each other: one can use~\cite{psenka2020second} for computing the curvature term, and the algorithms proposed in this paper for the other terms.

\section{Automatic differentiation (AD)}
In this section, we give a brief introduction to the automatic differentiation concept.
A reader familiar with this topic can skip this section.

AD is a technique for computing the value of the gradient of a smooth function $f\colon\mathbb{R}^{N} \to \mathbb{R}$ specified by a computer program.
In particular, it is assumed that $f$ can be represented as a sequence of elementary operations, for example additions, multiplications, trigonometric functions, logarithms, etc.
Evaluation of $f$ can also involve other operations such as matrix decompositions, for which differentiation formulas are available.
Under this assumption, AD allows for computing derivatives with working precision, and with the number of operations, which is only a small constant factor times larger than the number of operations to execute the evaluation of $f$ (i.e., with the same asymptotic complexity).

Let us illustrate the AD concept in a simple example.
Let $f\colon\mathbb{R}^2 \to \mathbb{R}$:
\[
	f(x_1, x_2) = e^{x_1 x_2} + \sin x_2,
\]
then it can be written as a sequence of elementary operations and depicted as the following computational graph:

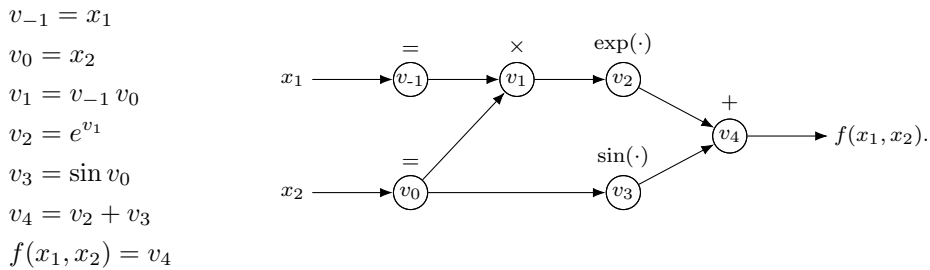
\begin{minipage}{0.3\textwidth}
\begin{equation*}
\begin{aligned}
	& v_{-1} = x_1 \\
	& v_0 = x_2 \\
	& v_1 = v_{-1}\, v_0 \\
	& v_2 = e^{v_{1}} \\
	& v_3 = \sin v_{0} \\
	& v_4 = v_2 + v_3 \\
	& f(x_1, x_2) = v_4
\end{aligned}
\end{equation*}
\end{minipage} 
\begin{minipage}{0.6\textwidth}
{\footnotesize
\begin{tikzpicture}[circ/.style={circle,draw,black,minimum size=0.45cm}]

\node [label=center:$x_2\text{\quad}$] (x2) at (0,0) {};
\node [label=center:$x_1\text{\quad}$] (x1) at (0,\y) {};

\node [circ, label=center:$v_0$] (v0) at (\x,0) {};
\node [circ, label=above:${=}$] at (\x,0) {};
\node [circ, label=center:$v_{\text{-}1}$] (vm1) at (\x,\y) {};
\node [circ, label=above:${=}$] at (\x,\y) {};

\node [circ, label=center:$v_{1}$] (v1) at (\x+\x,\y) {};
\node [circ, label=above:${\times}$] at (\x+\x,\y) {};

\node [circ, label=center:$v_{2}$] (v2) at (\x+\x+\x,\y) {};
\node [circ, label=above:${\mathrm{exp}(\cdot)}$] at (\x+\x+\x,\y) {};
\node [circ, label=center:$v_{3}$] (v3) at (\x+\x+\x,0) {};
\node [circ, label=above:${\mathrm{sin}(\cdot)}$] at (\x+\x+\x,0) {};

\node [circ, label=center:$v_{4}$] (v4) at (\x+\x+\x+\x,\y/2) {};
\node [circ, label=above:${+}$] at (\x+\x+\x+\x,\y/2) {};

\node [label=center:\qquad\qquad${f(x_1,x_2)}$.] (f) at (\x+\x+\x+\x+\x,\y/2) {};

\draw[-{Latex[scale=1.0]}] (x1) -- (vm1);
\draw[-{Latex[scale=1.0]}] (x2) -- (v0);
\draw[-{Latex[scale=1.0]}] (vm1) -- (v1);
\draw[-{Latex[scale=1.0]}] (v1) -- (v2);
\draw[-{Latex[scale=1.0]}] (v2) -- (v4);
\draw[-{Latex[scale=1.0]}] (v3) -- (v4);
\draw[-{Latex[scale=1.0]}] (v0) -- (v3);
\draw[-{Latex[scale=1.0]}] (v0) -- (v1);
\draw[-{Latex[scale=1.0]}] (v4) -- (f);

\end{tikzpicture}
}
\end{minipage}
\\

\noindent AD uses the chain rule\footnote{In this paper we focus on \emph{reverse-mode} autodiff, which is also sometimes called backpropagation. The alternative -- \emph{forward-mode} autodiff -- is typically used for functions $f:\mathbb{R}^{M} \to \mathbb{R}^{N}$ where $M < N$ because of the smaller asymptotic complexity in this case.} to find both components of $\nabla f$ in one pass through the computational graph in reverse order. Let $\overline{v}_i \triangleq \frac{\partial f}{\partial v_i}$ for $i=-1,\ldots,4$.
We have,
\begin{equation*}
\begin{aligned}
	&\overline{v}_4 = \frac{\partial f}{\partial v_4} = 1 \\
	&\overline{v}_3 = \frac{\partial f}{\partial v_4}\ \frac{\partial v_4}{\partial v_3} \equiv \overline{v}_4 \\
	&\overline{v}_2 = \frac{\partial f}{\partial v_4}\ \frac{\partial v_4}{\partial v_2} \equiv \overline{v}_4 \\
	&\overline{v}_1 = \frac{\partial f}{\partial v_2}\ \frac{\partial v_2}{\partial v_1} \equiv \overline{v}_2 e^{v_1}  \\
	&\overline{v}_0 = \frac{\partial f}{\partial v_3}\ \frac{\partial v_3}{\partial v_0} + \frac{\partial f}{\partial v_1}\ \frac{\partial v_1}{\partial v_0} \equiv \overline{v}_3 \cos v_0 + \overline{v}_1 v_{-1}\\
	&\overline{v}_{-1} = \frac{\partial f}{\partial v_1}\ \frac{\partial v_1}{\partial v_{-1}} \equiv \overline{v}_1 v_0 \\
\end{aligned}
\end{equation*}
where $\overline{v}_{-1} = \frac{\partial f}{\partial {v}_{-1}} \equiv \frac{\partial f}{\partial {x}_{1}}$ and $\overline{v}_{0} = \frac{\partial f}{\partial {v}_{0}} \equiv \frac{\partial f}{\partial x_2}$.

Thus, AD allows us to calculate all components of $\nabla f$ in one pass with $\mathcal{O}(F)$ complexity, where $F$ is the number of FLOP to calculate $f$ at a given $(x_1,\dots,x_N)$.
In general, the computational graph for computing the gradient of a function has as many nodes as the original graph for evaluating the function value, and each node is, at most, a small constant times more expensive than the corresponding node from the original graph.

Let us compare AD with numerical differentiation using finite differences, where components of a gradient of a function $f:\mathbb{R}^{N} \to \mathbb{R}$ are approximated, e.g., using forward differences  
\begin{equation}\label{eq:fd}
    \frac{\partial f}{\partial x_i}(x_1,\dots,x_N) \approx \frac{f(x_1,\dots,x_{i-1},x_i + h,x_{i+1}, \dots, x_N) - f(x_1,\dots,x_N)}{h},
\end{equation}
where $h$ is chosen so that the approximation error is small enough.
First, numerical differentiation is computationally more expensive than AD.
Indeed, \eqref{eq:fd} requires $N+1$ function evaluations to approximate $\nabla f$ and, hence, the complexity is $\mathcal{O}(N F)$.
Moreover, due to the error amplification of derivative approximation, \eqref{eq:fd} cannot achieve accuracy better than the square root of machine precision~\citep{gander2014scientific}.
At the same time, AD is more robust and can achieve machine precision accuracy~\citep{margossian2019review}.

Another alternative to AD and numerical differentiation is symbolic differentiation.
In it, one assembles the final formula for each component of the gradient using a sequence of rules as product rule, chain rule, etc. Since this constraint of expressing the entire result as a single formula does not allow introducing intermediate variables, in the worst case the final formula may contain exponentially many duplicated fragments.
By contrast to the symbolic differentiation, in AD one uses intermediate variables to define those duplicated fragments, allowing one to never evaluate any quantity more than once and providing efficiency guarantees.

For a more in-depth review of automatic differentiation see e.g.~\cite{griewank2008evaluating}.

\section{Riemannian optimization}\label{sec:riemannian-opt-briefer}

Let us briefly introduce the Riemannian optimization concept.
Let $\mathcal{M} \subset \mathbb{R}^{n_1\times \dots \times n_d}$ be a smooth embedded submanifold.
In this paper, we are concerned with the manifold of fixed-rank matrices ($d=2$) and the manifold of tensors of fixed tensor-train rank ($d>2$).
The definitions will be given in Section~\ref{sec:low-rank-matrix-manifold} and in Section~\ref{sec:TT_manifold} respectively.
In this section, we only provide an introductory overview without implementation details.

Our goal is to solve a minimization problem with a smooth function $f\colon\mathbb{R}^{n_1\times \dots \times n_d} \to \mathbb{R}$:
\[
\min_{\mX\in \mathbb{R}^{n_1\times \dots \times n_d}}
f(\mX).
\]
Assume that the solution to this problem can be approximated by a certain point $\mX_* \in \mathcal{M}$.
Then, we can reformulate the problem as
\begin{equation} \label{eq:optimization_main}
\min_{\mX\in \mathcal{M}}
f(\mX),
\end{equation}
i.e., the search space $\mathbb{R}^{n_1\times \dots \times n_d}$ is restricted to a Riemannian manifold~$\mathcal{M}$.
Riemannian optimization algorithms usually involve computation of Riemannian gradients $\mathrm{grad}\, f(\mX)$, which for embedded submanifolds of $\mathbb{R}^{n_1\times \dots \times n_d}$ and functions $f$ defined on the ambient space, may be written as a projection of the Euclidean gradient $\nabla f(\mX)$ to the tangent plane~$T_\mX \mathcal{M}$ of~$\mathcal{M}$ at the point $\mX$: 
\begin{equation} \label{eq:grad}
	\mathrm{grad}\, f(\mX) = \mathrm{P}_{T_\mX \mathcal{M}} \nabla f(\mX),
\end{equation}
where $\mathrm{P}_{T_\mX \mathcal{M}}\colon\mathbb{R}^{n_1\times \dots \times n_d} \to T_\mX \mathcal{M}$ denotes an operator of orthogonal projection to the tangent plane~$T_\mX \mathcal{M}$ and depends on $\mX$ non-linearly.
Given the Riemannian gradient notion, we may solve~\eqref{eq:optimization_main} using the Riemannian gradient descent
\[
\mX_{k+1} = R_{\mX_k}(\tau_k\, \mathrm{grad}\, f(\mX_k))),
\]
where $R_{\mX_k} \colon T_\mX \mathcal{M} \to \mathbb{R}^{n_1\times \dots \times n_d}$ returns a tangent vector back to the manifold (see~\cite{ao-retract-2014} for different retraction operations) and the parameter $\tau_k$ is chosen to ensure decay of the functional.
More advanced optimization algorithms, e.g., a Riemannian version of the conjugate gradient method is also available~\cite{absil}.

One can also utilize second-order methods, which involve computation of the Riemannian Hessian operator.
For the Riemannian Hessian $\mathrm{Hess} \, f(\mX)\colon T_\mX \mathcal{M} \to T_\mX \mathcal{M}$ we can use\footnote{Note that both $\mathrm{grad}$ and $\mathrm{Hess}$ operations depend on the particular choice of a manifold. Nevertheless, we do not use the subscript $\mathcal{M}$ as it will be clear from context and to not overcomplicate the notation.} the formula~\citep{absil2013extrinsic,kressner2016preconditioned}:
\begin{equation} \label{eq:hess}
	\mathrm{Hess} \, f(\mX)  = \mathrm{P}_{T_\mX \mathcal{M}}\nabla^2 f(\mX) + \mathrm{P}_{T_\mX \mathcal{M}} \dot{\mathrm{P}}_{T_\mX \mathcal{M}} (\nabla f(\mX)),
\end{equation}
where $\nabla^2 f(\mX)$ is the Euclidean Hessian and $\dot{\mathrm{P}}_{T_\mX \mathcal{M}}$ denotes the Fr\'echet derivative
of~$\mathrm{P}_{T_\mX \mathcal{M}}$.
The second term in~\eqref{eq:hess} arises due to the nonlinearity of the manifold.
For the manifold of low-rank matrices, it contains the inverse of a matrix of singular values~\citep{absil2013extrinsic}.
If singular values are small, this can lead to numerical instabilities.
To avoid this problem, the second term in~\eqref{eq:hess} can be omitted~\citep{kressner2016preconditioned,rakhuba2018jacobi}.
In this case, the optimization procedure can be interpreted as a constrained Gauss-Newton method.
We, therefore, consider only linearized Hessians and are interested in an efficient matrix-vector product by the first term of~\eqref{eq:hess}:
\begin{equation} \label{eq:ghess}
	\mathrm{H}_\mX [\mZ] \equiv \mathrm{P}_{T_\mX \mathcal{M}}\nabla^2 f(\mX) \, \mZ, \qquad \mX \in \mathcal{M}, \quad \mZ \in T_\mX \mathcal{M}.
\end{equation}

Note that first computing $\nabla f(\mX)$ as in~\eqref{eq:grad} and then applying $\mathrm{P}_{T_\mX \mathcal{M}}$ can be inefficient. %
Therefore, $\mathrm{P}_{T_\mX \mathcal{M}} \nabla f(\mX)$ should be calculated at once. For example, for the manifold of low-rank matrices, the Riemannian gradient $\mathrm{P}_{T_\mX \mathcal{M}} \nabla f(\mX)$ can always be represented as a low-rank matrix (see Sec.~\ref{sec:low-rank-matrix-manifold} for details), while the Euclidean gradient $\nabla f(\mX)$ can have an arbitrary large rank. Thus, using the Euclidean gradient in the intermediate calculations can lead to an inefficient algorithm. Similarly, first computing $\nabla^2 f(\mX) \, \mZ$ as in~\eqref{eq:ghess} and then applying  $\mathrm{P}_{T_\mX \mathcal{M}}$ can be significantly less efficient than calculating $\mathrm{P}_{T_\mX \mathcal{M}}\nabla^2 f(\mX) \, \mZ$ at once.
The goal of this paper is, thus, to develop an efficient tool to calculate~\eqref{eq:grad} and~\eqref{eq:ghess} -- the  building block operations of Riemannian optimization. 
The key assumption we make is that we can efficiently evaluate $f$ at any point~$\mX + T_\mX \mathcal{M}$, $\mX\in\mathcal{M}$.
Then, under mild conditions (see Propositions~\ref{thm:riemannian-grad-complexity} and~\ref{thm:riemannian-hess-complexity} and below), the overall complexity of the presented algorithm is only constant times larger than the complexity of the function evaluation.

Let us introduce the scalar product and the associated norm
\[
	\left<\mX, \mY\right> = \sum_{i_1,\dots,i_d = 1}^{n_1,\dots,n_d} \mX_{i_1,\dots,i_d} \mY_{i_1,\dots,i_d}, \qquad \|\mX \| = \left<\mX, \mX \right>^{1/2}.
\]
Using this notation, possible choices of $f(\mX)$ are, for example:
\begin{itemize}
	\item $f(\mX) = \|\mathrm{A}\mX - \mathbf{F}\|^2$ or $f(\mX) = \left<\mathrm{A}\mX, \mX\right> - 2 \left<\mathbf{F}, \mX\right>$ for given $\mathrm{A}\colon\mathbb{R}^{n_1\times \dots \times n_d}\to \mathbb{R}^{n_1\times \dots \times n_d}$ and  $\mathbf{F}\in \mathbb{R}^{n_1\times \dots \times n_d}$ that arise when solving linear systems; %
	\item $f(\mX) = {\left<\mathrm{A} [\mX], \mX\right>}/{\left<\mX, \mX\right>}$ with possibly nonlinear $\mathrm{A}\colon\mathbb{R}^{n_1\times \dots \times n_d}\to \mathbb{R}^{n_1\times \dots \times n_d}$, which arises when solving (nonlinear) eigenvalue problems;
	\item $f(\mX) = \|\mathrm{P}_\Omega (\mX - \ma)\|^2$ where $\mathrm{P}_\Omega$ denotes projection on the index set $\Omega$ such that %
	\[
		\mathrm{P}_\Omega \mX = 
		\begin{cases}
			X_{i_1 \dots i_d} \quad &(i_1,\dots,i_d) \in \Omega,\\
			0 & \text{otherwise}.
		\end{cases}
	\]
	This type of problem is referred to as matrix or tensor completion problems.
	\item $f(\mX)$ is a neural network loss function, which arises when using TT-decomposition to parametrize a recurrent neural network and applying Riemannian optimization for training  (for more details see~Section~\ref{sec:exp-functionals}).
\end{itemize}
In Section~\ref{sec:stopgrad}, we will also discuss how our approach can be used for operations that are not directly related to a minimization of a function, e.g., how to efficiently compute the preconditioned residual $\mathrm{P}_{T_\mX \mathcal{M}} \mathrm{B}^{-1} (\mathrm{A} \mX - \mathbf{F})$ for non-commuting $\mathrm{A}$ and $\mathrm{B}$.

\section{Automatic differentiation for the Riemannian gradient: fixed-rank matrices}
\label{sec:matr}
In this section, we propose an approach to automatically compute Riemannian gradients for the manifold of fixed-rank matrices. %

\subsection{The manifold of fixed-rank matrices} \label{sec:low-rank-matrix-manifold}

Let us briefly recall the concepts related to the manifold of fixed-rank matrices.
The set of matrices of fixed rank $r$: $\mathcal{M}_r = \{\mX\in\mathbb{R}^{m\times n}: \mathrm{rank}(\mX) = r\}$ forms a smooth submanifold of $\mathbb{R}^{m \times n}$~\cite[Example 8.14]{lee2003introduction}. 
Using SVD, any point $\mX \in \mathcal{M}_r$ of the manifold can be represented as $\mX = \mU \mS \mV^\intercal$, where $\mU \in \mathbb{R}^{m \times r}$ and $\mV \in \mathbb{R}^{n \times r}$ are matrices with orthonormal columns –– singular vectors ($\mU^\intercal \mU = \mI_r$, $\mV^\intercal \mV = \mI_r$) and $\mS \in \mathbb{R}^{r \times r}$ is the diagonal matrix of singular values. 
The tangent space $T_\mX \mathcal{M}_r$ of the manifold $\mathcal{M}_r$ at a point $\mX = \mU \mS \mV^\intercal \in \mathcal{M}_r$ can be written as
\begin{equation}
\label{eq:matrix_tangent_space}
T_\mX \mathcal{M}_r = \{\mdu \mV^\intercal + \mU \mdv^\intercal \, \mid \, \mdu \in \mathbb{R}^{m \times r}, \mdv \in \mathbb{R}^{n \times r}: \mV^\intercal \mdv = \mathbf{O}_{r\times r}\},
\end{equation}
where $\mathbf{O}_{r\times r}$ denotes a zero matrix of size $r\times r$. In what follows, we refer to the matrices $\mdu$ and $\mdu$ that define an element of the tangent space as \emph{delta-matrices}.
The orthogonal projection of $\mZ\in\mathbb{R}^{m\times n}$ to the tangent space $T_\mX \mathcal{M}_r$ can, thus, be obtained as follows:
\begin{equation}
\label{eq:projection_low_rank}
\mathrm{P}_{T_\mX \mathcal{M}_r} \mZ = \mZ \mV \mV^\intercal + \mU \mU^\intercal \mZ (\mI - \mV \mV^\intercal).
\end{equation}
We refer the reader to, e.g.,~\cite[Sec. 2.1]{vandereycken2013low} for a more detailed discussion of the manifold of low-rank matrices, including the derivation of  \eqref{eq:projection_low_rank}.

Finally, to simplify the notation, we denote the projection operator as 
\[
	\mathrm{P}_\mX \triangleq \mathrm{P}_{T_\mX \mathcal{M}_r}.
\]
We also introduce $\mathcal{T}_{\mX}$ that maps parametrization matrices to an element of the tangent plane at the point $\mX$
\begin{equation}
\label{eq:matrix-tangent-mapping}
  \mathcal{T}_{\mX}: \mathbb{R}^{m \times r} \times \mathbb{R}^{n \times r} \to T_\mX \mathcal{M}_r,
\end{equation}
namely,
\begin{equation}\label{eq:tangent-matrix-parametrization}
  \mathbf{T} = \mathcal{T}_{\mX} (\mdu, \mdv) = \mdu \mV^\intercal + \mU \mdv^\intercal.
\end{equation}
This mapping will be used later in Sec.~\ref{sec:low-rank-matrix-autodiff} to simplify the presentation of the algorithm.

\subsection{Automatic differentiation approach}\label{sec:low-rank-matrix-autodiff}

In this section, we propose an efficient way of computing the Riemannian gradient 
\[
	\mathrm{grad}\, f(\mX) = \mathrm{P}_\mX\, \nabla f \in T_{\mX} \mathcal{M}_r \subset \mathbb{R}^{m \times n}.
\]
The Riemannian gradient $\mathrm{grad}\, f(\mX)$ is an $m \times n$ matrix, but as noted in the previous section it can be defined via the delta matrices $\mdu$ and $\mdv$ using just $(m + n)r$ parameters ($(m + n)r - r^2$ if gauge condition $\mV^\intercal \mdv = \mathbf{O}_{r\times r}$ are taken into account). Thus, if we can avoid using full $m \times n$ matrices in intermediate calculations, we can potentially compute the Riemannian gradient with a better asymptotic complexity than $\compl(mn)$.

A naive approach of computing the Riemannian gradient is to first compute ${\partial f}/{\partial \mX}$ with AD and then project the result to the tangent plane by using formula~\eqref{eq:projection_low_rank}:
\begin{equation}\label{eq:proj_grad}
	\mathrm{P}_{\mX} \nabla f = \nabla f \mV \mV^\intercal + \mU \mU^\intercal \nabla f\, (\mI - \mV \mV^\intercal).
\end{equation}
The problem with this approach is that it requires finding the full matrix of the Euclidean gradient ${\partial f}/{\partial \mX}$ of the size $m\times n$, which we want to avoid.
Alternatively, we may find the Riemannian gradient without explicitly forming ${\partial f}/{\partial \mX}$.
In particular, we notice that the Riemannian gradient~\eqref{eq:proj_grad} involves computing the following multiplication of matrices
\begin{equation}
    \label{eq:matrix-autodiff-quantities-of-interest}
	\left(\nabla f \mV\right) \in \mathbb{R}^{m\times r}, \quad \left(\mU^\intercal \nabla f\right) \in \mathbb{R}^{r\times n}.
\end{equation}
We may find these two quantities by using the classical AD as follows:
\[
\nabla f \mV = \nabla_{\mE} f(\mE \mV^\intercal)|_{\mE=\mU \mS}, \quad \mU^\intercal \nabla f = \nabla_{\mF} f(\mU \mF)|_{\mF=\mS \mV^\intercal}.
\]
So, one can use classic AD on the function $f$ \emph{twice} (each time with the complexity equal to evaluating the function $f$ at a single point due to AD properties) to compute all the pieces that depend on $f$.

However, in the rest of this section we propose an alternative way of computing quantities~\eqref{eq:matrix-autodiff-quantities-of-interest} by using classic AD a \emph{single time} on a specially introduced auxiliary function. This alternative approach is introduced because it naturally generalizes into an efficient algorithm for the tensor case (see Sec.~\ref{sec:riemannian-grad-tensors}).

Quantities~\eqref{eq:matrix-autodiff-quantities-of-interest} can be computed at once by differentiating (using AD) the following auxiliary function defined using mapping~\eqref{eq:tangent-matrix-parametrization}
\[
  g \myeq f\circ \mathcal{T}_\mX.
\]
We have 
\begin{equation}
\label{eq:matrix_aux_def}
	g(\mA, \mB) = f(\mathcal{T}_\mX(\mA, \mB)) = f(\mA \mV^\intercal + \mU \mB^\intercal).
\end{equation}
Indeed, $\mX = \mU \mS \mV^\intercal$ can be represented as
\[
	 \mathcal{T}_\mX(\mU\mS, \mathbf{O}_{n\times r}) = 
	(\mU\mS)\cdot\mV^\intercal + \mU\cdot\mathbf{O}_{n\times r}^\intercal = \mX,
\]
and, hence, the partial derivatives of $\mathbf{T} = \mathcal{T}_{\mX} (\mdu, \mdv)$ at $(\mA, \mB) = (\mU\mS,\mathbf{O}_{n\times r})$  are
\[
  \frac{\partial T_{ij}}{\partial A_{pq}}
  = \frac{\partial (\mA \mV^\intercal + \mU \mB^\intercal)_{ij}}{\partial A_{pq}} = \delta_{ip}V_{jq},
\]
\[
  \frac{\partial T_{ij}}{\partial B_{pq}}
  = \frac{\partial (\mA \mV^\intercal + \mU \mB^\intercal)_{ij}}{\partial B_{pq}} = \delta_{jp}U_{iq}.
\]
where $\delta_{ip}$ is the Kronecker delta.
Applying the chain rule to~\eqref{eq:matrix_aux_def}, we get
\begin{equation*}
\begin{aligned}
	& \left.\frac{\partial g}{\partial A_{pq}}\right|_{\subalign{&\mA = \mU\mS, \\ &\mB = \mathbf{O}_{n\times r}}} = \sum_{i,j} \left.\frac{\partial f}{\partial T_{ij}}\right|_{\mT = \mX} \left.\frac{\partial T_{ij}}{\partial A_{pq}}\right|_{\subalign{&\mA = \mU\mS, \\ &\mB = \mathbf{O}_{n\times r}}} = \sum_{i,j} \frac{\partial f}{\partial X_{ij}} \delta_{ip}V_{jq} = \left(\nabla f \mV\right)_{pq}, \\[5pt]
    & \left.\frac{\partial g}{\partial B_{pq}}\right|_{\subalign{&\mA = \mU\mS, \\ &\mB = \mathbf{O}_{n\times r}}} = \sum_{i,j} \left.\frac{\partial f}{\partial T_{ij}}\right|_{\mT = \mX} \left.\frac{\partial T_{ij}}{\partial B_{pq}}\right|_{\subalign{&\mA = \mU\mS, \\ &\mB = \mathbf{O}_{n\times r}}} = \sum_{i,j} \frac{\partial f}{\partial X_{ij}} \delta_{jp}U_{iq} = \left(\mU^\intercal \nabla f\right)_{qp}.
\end{aligned}
\end{equation*}
Thus, a low-rank representation of the Riemannian gradient can be written as 
\[
	\mathrm{P}_{\mX} \nabla f = 
	\begin{bmatrix}
		\mU & \mdu
	\end{bmatrix}
	\begin{bmatrix}
		\mdv & \mV
	\end{bmatrix}^\intercal,
\]
with 
\begin{equation}
\begin{aligned}
\label{eq:matrix_autodiff}
\mdu &= \left.\frac{\partial g}{\partial \mA}\right|_{\subalign{&\mA = \mU\mS, \\ &\mB = \mathbf{O}_{n\times r}}} \\
\mdv^\intercal &= \left.\frac{\partial g}{\partial \mB^\intercal}\right|_{\subalign{&\mA = \mU\mS, \\ &\mB = \mathbf{O}_{n\times r}}}  \left(\mI - \mV \mV^\intercal \right).
\end{aligned}
\end{equation}
The algorithm to compute the Riemannian gradient is summarized in Algorithm~\ref{alg:matrix-riemannian-autodiff}.

\begin{algorithm}[t]
\caption{Computing the Riemannian gradient for low-rank matrices via AD.}\label{alg:matrix-riemannian-autodiff}
\begin{algorithmic}[1]
\Require $\mX = \mU\mS\mV^\intercal \in\mathbb{R}^{m \times n}$, $p(\mathbf{L}, \mathbf{R})$ -- implementation of evaluating $f$ at $\mathbf{L}\mathbf{R}^\intercal$ for any $\mathbf{L}\in\mathbb{R}^{m\times 2r}$ and $\mathbf{R}\in\mathbb{R}^{n\times 2r}$.
\Ensure $\mdu, \mdv$ such that $\mathrm{P}_\mX \nabla f = \mdu \mV^\intercal + \mU \mdv^\intercal$
\Statex
\Function{g}{$\mA, \mB$}
  \State \Return $p([\mU \, \mA],\, [\mB\, \mV])$
\EndFunction
\State Using AD, compute $\mdu := \frac{\partial g}{\partial \mA}|_{\subalign{&\mA = \mU\mS, \\ &\mB = \mathbf{O}_{n\times r}}}$
\State Using AD, compute $\mdv := \frac{\partial g}{\partial \mB}|_{\subalign{&\mA = \mU\mS, \\ &\mB = \mathbf{O}_{n\times r}}}$
\State $\mdv^\intercal := \mdv^\intercal  - (\mdv^\intercal \mV) \mV^\intercal $
\end{algorithmic}
\end{algorithm}

\subsection{Complexity of the approach}\label{sec:riemannian-grad-mat-complexity}
Let us estimate the complexity of computing $\mdu$ and $\mdv$ by the proposed approach, i.e., by defining the auxiliary function $g$ and differentiating it with respect to $\mA$ and $\mB$.

\begin{statement}
\label{thm:matrix-riemannian-grad-complexity}
Let $f: \mathbb{R}^{m \times n} \rightarrow \mathbb{R}$ be a smooth function defined by a program $p$, which takes as input SVD decomposition of a matrix $\mX = \mU\mS\mV^\intercal \in\mathbb{R}^{m\times n}$ and outputs the value $f(\mX)$ in $F = F(m,n,r)$ floating point operations (FLOP), which is polynomial with respect to the rank of the matrix $\tens{X}$ (i.e., the program $p$ belongs to the P complexity class). Then, the complexity of using Alg.~\ref{alg:matrix-riemannian-autodiff} for computing delta terms $\mdu$ and $\mdv$ which define the Riemannian gradient $\mathrm{P}_\mX \nabla f = \mdu \mV^\intercal + \mU \mdv^\intercal$
   is $\compl(F + n \rank^2)$.
\end{statement}

As an example, computing 
\[
	f(\mX) = \left<\mX, \mX \right>, \quad \mX = \mU \mV^\intercal, \quad \mU\in\mathbb{R}^{m\times r}, \mV\in\mathbb{R}^{n\times r},
\]
leads to $F = \mathcal{O} \left( (n+m)r^2 \right)$ FLOP, since
\[
	\left<\mX, \mX \right> = \mathrm{trace}(\mU\mV^\intercal \mV\mU^\intercal) = \mathrm{trace}\left((\mU^\intercal\mU) (\mV^\intercal \mV) \right).
\]

\begin{proof}[Proof of Prop.~\ref{thm:matrix-riemannian-grad-complexity}]
The auxiliary function $g(\mA, \mB)$ can be constructed by feeding to $p$ the factors of the matrix 
\[
	\mA \mV^\intercal + \mU \mB^\intercal = \begin{bmatrix}\mA&\mU\end{bmatrix} \begin{bmatrix}\mV &\mB\end{bmatrix}^\intercal
\]
which is represented with the rank $2r$ --- twice larger than the rank $r$ of the original matrix. As a result, the asymptotic complexity (as a function of $n$ and $r$) of evaluating the function on such a matrix is still $\compl(F)$.
Thanks to the properties of AD, computing the derivatives of $g$ with respect to the factors~$\mA$ and~$\mB$ has the same complexity $\compl(F)$. 
Finally, computing the factor $\mdv$ using~\eqref{eq:matrix_autodiff} can be done in $\compl(r^2 n)$, yielding the total complexity $\compl(F + n \rank^2)$.
\end{proof}

For most functions used in practice, the asymptotic complexity $F$ of executing the function at one point exceeds $\compl(n \rank^2)$ and the total complexity of the proposed algorithm (as a function of $n$ and~$r$) equals to $\compl(F + n \rank^2) = \compl(F)$.

\subsection{More general view of the proposed algorithm}  \label{sec:intuitive-general}

In this section, we look at the proposed algorithm from a more general perspective, trying to avoid specifics of the fixed-rank manifold. The main idea of the proposed algorithm is to introduce the auxiliary function~\eqref{eq:auxfunc} and express the desired Riemannian gradient $\mathrm{grad}\, f(\vec{X})$ in terms of its derivatives~\eqref{eq:matrix_aux_def}. 
Note that we could have used an alternative auxiliary function\footnote{One can use the same derivation as in Sec.~\ref{sec:low-rank-matrix-autodiff} to prove that differentiating the alternative auxiliary function $\widehat{h}$ yields the same results.} 

\[h(\mat{C}, \mat{D}) = f(\mX + \mat{C} \mV^\intercal + \mU \mat{D}^\intercal) = f((\mU \mS + \mat{C}) \mV^\intercal + \mU \mat{D}^\intercal) = g(\mat{C} + \mU \mS, \mat{D}).\] 
If one combines both arguments of the mapping $\mathcal{T}_\mX(\mA, \mB)$ (see~\eqref{eq:matrix-tangent-mapping}) into a single $(n + m) r$ dimensional vector $\vec{v}$, you can define an equivalent mapping
$
  \widehat{\mathcal{T}}_{\mX}\colon \mathbb{R}^{(n + m) r} \to T_\mX \mathcal{M}_r
$

and an alternative representation of the auxiliary function 
\begin{equation}\label{eq:hfun}
    \widehat{h}(\vec{v}) = f(\mX + \widehat{\mathcal{T}}_{\mX}(\vec{v})).
\end{equation}

Thus, the proposed approach is equivalent to defining a mapping $\widehat{\mathcal{T}}_{\mX}$ from the parametrization of the tangent space onto the tangent space itself, defining an auxiliary function $\widehat{h}(\vec{v})$~\eqref{eq:hfun}, computing its gradient using classical AD, and finally doing certain post processing of this gradient: reshaping, enforcing the gauge conditions as in~\eqref{eq:matrix_autodiff}. %
It is not surprising that we obtain a Riemannian gradient using these formulas, as informally the gradient of the auxiliary function \eqref{eq:hfun} is the fastest ascent direction of $\widehat{h}(\vec{v})$ at $\vec{v} = 0$ and, hence, of $f$ at $\mX$ in the direction of all possible vectors from $T_\mX \mathcal{M}_r$.

\section{Automatic differentiation for the Riemannian gradient: fixed-rank tensors}
\label{sec:tens}
In this section, we extend the results of Sec.~\ref{sec:low-rank-matrix-autodiff} to fixed-rank tensor-train tensors, which is a generalization of fixed-rank matrices to multidimensional arrays.

\subsection{The manifold of TT-tensors of fixed rank} \label{sec:TT_manifold}
A tensor $\tens{A} \in \mathbb{R}^{n_1 \times \ldots \times n_d}$ is said to be represented in the \emph{tensor-train} format~\citep{oseledets2011ttMain} (\emph{TT-format}) if each of its elements $A_{i_1 \ldots i_d}$ is a product of $d$ matrices:
\begin{equation}
\label{eq:TT}
A_{i_1 \ldots i_d} = \mG_1[i_1] \ldots \mG_d [i_d],
\end{equation} 
where for fixed $i_k = 1, \ldots, n_k$, $\mG_k[i_k]$ is an $r_{k-1} \times r_{k}$ matrix for any value of $k = 1, \ldots, d$. 
We require $r_0 = r_d = 1$ such that $\mG_1[i_1]$ is $1\times r_1$ row vector and $\mG_d [i_d]$ is $r_{d-1} \times 1$ column vector. 
The three-dimensional arrays $\mG_k$ of sizes $r_{k-1} \times n_k \times r_k$, $k=1,\dots,d$ are called \emph{TT-cores} and the vector
\[
\mathbf{r}_{\mathrm{TT}}(\tens{A}) = (r_1,\dots,r_{d-1}),
\]
is called the \emph{TT-rank} of $\mA$. 
For a more detailed discussion on the properties of the TT-format see~\cite{oseledets2011ttMain}.

Like in the matrix case (Sec.~\ref{sec:matr}), the set of tensors  \[\mathcal{M}_\mathbf{r} = \{\tens{A} \in \mathbb{R}^{n_1 \times \ldots \times n_d} \,\mid\, \mathbf{r}_{\mathrm{TT}}(\tens{A}) = \mathbf{r}\}.\]
 forms a smooth manifold.
To parametrize its tangent spaces, we need the notion of orthogonalization of the TT-cores. 
A TT-representation~\eqref{eq:TT} is called $\mu$-orthogonal, $\mu = 2,\dots,d-1$ if 
\begin{equation}
\label{eq:left_orthogonal}
	\sum_{i_k=1}^{n_k}\mG_k[i_k]^\intercal \mG_k[i_k] = \mI_{r_k}, 
\end{equation}
for $k = 1, \ldots, \mu - 1$, and 
\begin{equation}
\label{eq:right_orthogonal}
	\sum_{i_k=1}^{n_k}\mG_k[i_k] \mG_k[i_k]^\intercal = \mI_{r_{k-1}}
\end{equation}
for $k = \mu + 1, \ldots, d$.
If $\mu=1$ or $\mu=d$, we only require~\eqref{eq:right_orthogonal} or~\eqref{eq:left_orthogonal} respectively.
The cores satisfying~\eqref{eq:left_orthogonal} and~\eqref{eq:right_orthogonal} are called, respectively, left- and right-orthogonal cores.
TT-decomposition of a tensor is not unique, and for any $\mu = 1, \ldots, d$ there exists a $\mu$-orthogonal representation of a given tensor~\cite[Sec.~4.2.1]{steinlechner2016riemannian}. 
Moreover, for any $1\leq\mu_1 \leq \mu_2 \leq d$, the $\mu_1$-orthogonal and $\mu_2$-orthogonal decompositions can be constructed to share the left-orthogonal TT-cores $\mG_1, \ldots, \mG_{\mu_1 - 1}$ satisfying~\eqref{eq:left_orthogonal} and the right-orthogonal TT-cores $\mG_{\mu_2 + 1}, \ldots, \mG_d$ satisfying~\eqref{eq:right_orthogonal}. 

For a given tensor $\tens{X}$, one can define a set of left-orthogonal TT-cores $\mU_1, \ldots, \mU_{d-1}$, right-orthogonal TT-cores $\mV_2, \ldots, \mV_d$, and unrestricted TT-cores $\mS_1, \ldots, \mS_d$ such that for any $\mu=1,\dots,d$, there exists the following $\mu$-orthogonal decomposition of the tensor %
\begin{equation}
\label{eq:orthogonalized-tt}
X_{i_1 \ldots i_d} = \mU_1[i_1] \ldots \mU_{\mu - 1}[i_{\mu-1}] \mS_\mu[i_\mu] \mV_{\mu + 1}[i_{\mu+1}] \ldots \mV_d[i_d].
\end{equation}
Using the left-orthogonal TT-cores $\mU_1, \ldots, \mU_{d-1}$ and the right-orthogonal TT-cores $\mV_2, \ldots, \mV_d$ of tensor $\mX\in  \mathcal{M}_\mathbf{r}$, one may parametrize the tangent space $T_{\mX} \mathcal{M}_r$ as follows

\begin{equation} 
\label{eq:rank-d-tangent-space}
\begin{split}
  T_{\mX} \mathcal{M}_r = \Bigl\{ &
  \mathbf{T}\in \mathbb{R}^{n_1\times\dots\times n_d}\colon T_{i_1\dots i_d} = 
   \dot{\mS}_1[i_1]\, \mV_2[i_2]\, \ldots \mV_d[i_d] \,+ 
   \mU_1[i_1]\dot{\mS}_2[i_2]\, \mV_3[i_3] \ldots \mV_d[i_d]  \\
  &+
  \dots +\mU_1[i_1]\, \ldots  \mU_{d-1}[i_{d-1}] \, \dot{\mS}_d[i_d], 
  \ 
       \dot{\mS}_k \in \mathbb{R}^{r_{k-1} \times n_k \times r_k}, \  k=1,\dots,d,\ r_0=r_{d} = 1
\Bigr\}.
\end{split}
\end{equation}

In what follows, we refer to the tensors $\dot{\mS}_1, \ldots, \dot{\mS}_d$ that define an element of the tangent space as \emph{delta-terms}.

Additional gauge conditions are usually introduced\footnote{These gauge conditions generalize the orthogonality constraint $\mV^\intercal \mdv = \mathbf{O}_{r\times r}$ in the definition of the matrix tangent space~\eqref{eq:matrix_tangent_space} to the tensor case.} to uniquely parametrize elements of the tangent space:
\begin{equation}\label{eq:tgauge}
	\sum_{i_k=1}^{n_k} \mU_k[i_k]^\intercal \, \dot{\mS}_k [i_k] = 0, \quad k = 1, \ldots, d-1.
\end{equation}
In what follows, we always assume that the deltas $\dot{\mS}_1, \dot{\mS}_2, \dots, \dot{\mS}_d$ that define an element of the tangent space obey the gauge conditions~\eqref{eq:tgauge}.

Note that in~\eqref{eq:rank-d-tangent-space}, the expression for an element of the tangent space is formally represented as a sum of $d$ TT-tensors, each of TT-rank $\mathbf{r}$, and, hence, can be represented as a single TT tensor with the TT-rank $\mathbf{r} + \dots + \mathbf{r} = d\mathbf{r}$~\cite[Sec.~4.1]{oseledets2011ttMain}.
Nevertheless, thanks to the common cores, it can be represented with the TT-rank equal to $2 \mathbf{r}$. 
Indeed, by directly multiplying the block matrices, one can verify that
\begin{equation}
\label{eq:tt_deltas_to_tangent}
T_{i_1 \ldots i_d} = \big [ \dot{\mS}_1[i_1] ~ \mU_1[i_1] \big]
\begin{bmatrix}
\mV_2[i_2] & \\ 
\dot{\mS}_2[i_2] & \mU_2[i_2]
\end{bmatrix}
\ldots 
\begin{bmatrix}
\mV_{d-1}[i_{d-1}] & \\ 
\dot{\mS}_{d-1}[i_{d-1}] & \mU_{d-1}[i_{d-1}]
\end{bmatrix}
\begin{bmatrix}
\mV_d[i_{d}]\\ 
\dot{\mS}_d[i_{d}]
\end{bmatrix}.
\end{equation}
For convenience, we also introduce a function that maps the delta terms $\dot{\mS}_k$ to an element of the tangent space
\[	
	\mathcal{T}_{\mX}: \mathbb{R}^{1\times n_1 \times r_1} \times \mathbb{R}^{r_{1} \times n_2 \times r_2} \times \dots \times \mathbb{R}^{r_{d-2} \times n_{d-1} \times r_{d-1}} \times \mathbb{R}^{r_{d-1} \times n_d \times 1} \to T_{\mX} \mathcal{M}_r,
\]
namely
\begin{equation}\label{eq:parametr}
	\mathbf{T} = \mathcal{T}_{\mX} (\dot{\mS}_1, \dots, \dot{\mS}_d),
\end{equation}
as is defined in~\eqref{eq:tt_deltas_to_tangent}.
The following proposition gives the explicit representation of a general tensor projected onto the tangent plane of $\mathcal{M}_{\mathbf{r}}$.
\begin{statement}
\cite[equation~(4.17)]{steinlechner2016riemannian}
The orthogonal projection $\mathrm{P}_{\mX} \tens{Z}$ of a given tensor $\tens{Z}\in\mathbb{R}^{n_1\times\dots\times n_d}$ onto the tangent space $T_{\mX} \mathcal{M}_{\mathbf{r}}$ is defined as an element of the tangent space~\eqref{eq:rank-d-tangent-space} with $\dot{\mS}_k$:

\begin{equation}
\begin{aligned}
\label{eq:tt-projection}
\underbrace{\dot{\mS}_k [j_k]}_{r_{k-1} \times r_k} = \sum_{i_1, \ldots, i_d} & \left(\underbrace{\mU_1[i_1] \ldots \mU_{k-1}[i_{k-1}] \left(\mI_{r_{k-1} }\delta_{j_k i_k} - \mU_{k}[j_{k}] \mU_{k}[i_{k}]^\intercal\right)}_{1 \times r_{k-1}}\right)^\intercal Z_{i_1 \ldots i_d}\\
&\left(\underbrace{\mV_{k+1}[i_{k+1}] \ldots \mV_{d}[i_d]}_{r_k \times 1}\right)^\intercal, \quad k=1,\dots,d-1
\end{aligned}
\end{equation}
and $\dot{\mS}_d$ as
\begin{equation*}
\begin{aligned}
{\mS}[i_d] = \sum_{i_1, \ldots, i_{d-1}} \mU_1[i_1] \ldots \mU_{d-1}[i_{d-1}] Z_{i_1 \ldots i_d}.
\end{aligned}
\end{equation*}
\end{statement}
For a more detailed discussion of the manifold of fixed tensor-train rank tensors (including derivations of the above equations) see, e.g., Sec.~4.3-4.4 of~\cite{steinlechner2016riemannian}.

\subsection{Automatic differentiation} \label{sec:riemannian-grad-tensors}
Let us find the Riemannian gradient of a function $f: \mathbb{R}^{n_1\times \dots \times n_d} \rightarrow \mathbb{R}$ at a point $\tens{X}$.
Similarly to the matrix case, we consider an auxiliary function using~\eqref{eq:parametr}:
\[
	g \myeq f\circ \mathcal{T}_\mX.
\]
Note that the intuitive explanation of the proposed method provided in Sec.~\ref{sec:intuitive-general} still applies in this case.

In particular, we have 
\[
	\mT  = \mathcal{T}_\mX (\mS_1, \mathbf{O}_2,\dots, \mathbf{O}_d) = \mX,
\]
where $\mathbf{O}_k$, $k=2,\dots,d$ are zero tensors of appropriate sizes and $\mS_1$ is defined in~\eqref{eq:orthogonalized-tt} for $\mu=1$. 
As a result,
\[
	g(\mS_1, \mathbf{O}_{2},\dots, \mathbf{O}_{d}) = f\left(\mX\right).
\]
Consider the derivative of $g(\mR_1,\dots,\mR_d)$ with respect to $\mR_k$ at a point $\mathcal{R}_0 = (\mS_1, \mathbf{O}_{2},\dots, \mathbf{O}_{d})$:
\begin{equation}
\begin{aligned}
\label{eq:core-derivative}
\frac{\partial g}{\partial \mR_k[i_k]}(\mathcal{R}_0) &= 
\sum_{i_1, \ldots, i_{k-1}, i_{k+1}, \ldots, i_d} \frac{\partial f}{\partial {T}_{i_1 \ldots i_d}} (\mX)\, \frac{\partial {T}_{i_1 \ldots i_d}}{\partial \mR_k[i_k]} ({\mathcal{R}_0}) \\
&=\sum_{i_1, \ldots, i_{k-1}, i_{k+1}, \ldots, i_d} \left(\mU_1[i_1] \ldots \mU_{k-1}[i_{k-1}]\right)^\intercal \,\frac{\partial f}{\partial \tensel{X}_{i_1 \ldots i_d}} \left(\mV_{k+1}[i_{k+1}] \ldots \mV_d[i_{d}]\right)^\intercal.
\end{aligned}
\end{equation}
By comparing expressions~\eqref{eq:tt-projection} and~\eqref{eq:core-derivative}, it is easy to see that the $\dot{\mS}_k$ that defines the Riemannian gradient $\mathrm{P}_\mX \nabla f$ can be computed as
\begin{equation}
\begin{aligned}
\label{eq:tt-enforcing-gauge-conditions}
\dot{\mS}_k[i_k] &= \frac{\partial g}{\partial \mR_k[i_k]}(\mathcal{R}_0) - \mU_{k}[i_{k}] \sum_{j_k} \mU^\intercal_{k}[j_{k}]\frac{\partial g}{\partial \mR_k[j_k]}(\mathcal{R}_0) \quad  k = 1, \ldots, d-1,\\
\dot{\mS}_d[i_d] &= \frac{\partial g}{\partial \mR_d[i_d]}(\mathcal{R}_0).
\end{aligned}
\end{equation}

\begin{algorithm}[t]
\caption{Converting delta notation to TT-cores (implementation of~\eqref{eq:parametr}).}\label{alg:deltas-to-tangent-space}
\begin{algorithmic}[1]
\Require TT-tensor $\mX$ defined by the TT-cores $\mG_k$, tensors $\dot{\mS}_k$ that define the tangent space element $\mathbf{T}\in T_\mX \mathcal{M}_\mathbf{r}$ (see~\eqref{eq:tt_deltas_to_tangent})
\Ensure $\widehat{\mG}_k$, $k=1,\dots,d$ –– TT-cores of $	\mathbf{T} = \mathcal{T}_{\mX} (\dot{\mS}_1, \dots, \dot{\mS}_d)$ %
\State Compute resp. left- and right-orthogonal $\{\mU_k\}_{k=1}^{d-1}$ and $\{\mV_k\}_{k=2}^{d}$, and tensors $\{\mS_k\}_{k=1}^d$ as in~\eqref{eq:orthogonalized-tt} 
\For{$i_1 = 1$ to $n_1$}
\State $\widehat{\mG}_1[i_1] =
\big [ \dot{\mS}_1[i_1] ~ \mU_1[i_1] \big]$\;
\EndFor
\For{$k = 2$ to $d-1$}
\For{$i_k = 1$ to $n_k$}
\State $\widehat{\mG}_k[i_k] = 
\begin{bmatrix}
	\mV_{k}[i_{k}] &  \\ 
	\dot{\mS}_k[i_{k}] & \mU_{k}[i_{k}]
\end{bmatrix}$\;
\EndFor
\EndFor
\For{$i_d = 1$ to $n_d$}
\State $\widehat{\mG}_d[i_d] =
\begin{bmatrix}
	\mV_d[i_{d}]\\ 
	\dot{\mS}_d[i_{d}]
\end{bmatrix}$\;
\EndFor
\end{algorithmic}
\end{algorithm}

\begin{algorithm}[t]
\caption{Computing the Riemannian gradient for low-rank tensors via AD. %
}\label{alg:tensor-riemannian-autodiff}
\begin{algorithmic}[1]
\Require $\{\mG_k\}_{k=1}^d$ -- TT-cores of $\mX$, $p(\widehat\mG_1, \ldots, \widehat\mG_d)$ -- Python implementation of $f(\widehat\mX)$ for a point $\widehat\mX$ given by TT-cores $\widehat\mG_1, \ldots, \widehat\mG_d$.
\Ensure The TT-cores $\{\mgrad_k\}_{k=1}^d$ of the Riemannian gradient $\mathrm{grad} \, f(\mX)$ 
\State For $\mX$, compute resp. left- and right-orthogonal $\{\mU_k\}_{k=1}^{d-1}$, $\{\mV_k\}_{k=2}^{d}$  and $\{\mS_k\}_{k=1}^d$ as in~\eqref{eq:orthogonalized-tt}.
\Function{g}{$\mR_1, \ldots, \mR_d$}
  \State Run Alg.~\ref{alg:deltas-to-tangent-space} passing as input $\{\mG_k\}_{k=1}^d$, $\{\mR_k\}_{k=1}^d$ and write the output into~$\{\widehat{\mG}_k\}_{k=1}^d$\;
  \State \Return $p(\widehat{\mG}_1, \ldots, \widehat{\mG}_d)$
\EndFunction
\State Using AD, compute $\dot{\mS}_k := \left.\frac{\partial g}{\partial \mR_k}\right|_{(\mR_1, \mR_2, \ldots, \mR_d) = (\mS_1, \mathbf{O}_2, \ldots, \mathbf{O}_d)}$ for $k = 1, \ldots, d$\;%
\For{$k \gets 1$ to $d-1$}
\State $\tens{D}_k := \reshape(\dot{\mS}_k,\, (r_{k-1} n_k,\, r_k))$
\State $\mU^\mathrm{L}_{k} := \reshape(\mU_{k},\, (r_{k-1} n_k,\, r_k))$
\State $\tens{D}_k := \tens{D}_k + \mU^\mathrm{L}_{k}\left(\left(\mU^\mathrm{L}_{k}\right)^\intercal\, \tens{D}_k\right)$\; 
\Comment{See~\eqref{eq:tt-enforcing-gauge-conditions}. Parentheses indicate the order of operations}
\State $\dot{\mS}_k := \reshape(\tens{D}_k, (r_{k-1}, n_k, r_k))$
\EndFor
\State Run Alg.~\ref{alg:deltas-to-tangent-space} passing as input $\{\mG_k\}_{k=1}^d$ and $\{\dot{\mS}_k\}_{k=1}^d$ and write the output TT-cores into $\{\mgrad_k\}_{k=1}^d$\;
\end{algorithmic}
\end{algorithm}

The algorithm for computing the Riemannian gradient in the tensor-train case is listed in Alg.~\ref{alg:tensor-riemannian-autodiff}.
Hereinafter we use a \texttt{reshape}~\cite{oliphant2006guide} function that changes the shape of an array, preserving the values and the order of elements, where by the order of elements of $\mX\in\mathbb{R}^{n_1\times\dots\times n_d}$ we imply the following ordering:
\[
   (i_1,\dots, i_d) \mapsto 1 + \sum_{\alpha=1}^d (i_\alpha -1) \prod_{\beta = \alpha+1}^{d} n_\beta.
\]

Let us estimate the complexity of Alg.~\ref{alg:tensor-riemannian-autodiff}

\begin{statement}
\label{thm:riemannian-grad-complexity}
Let $f: \mathbb{R}^{n_1 \times \ldots \times n_d} \rightarrow \mathbb{R}$ be a smooth function defined by a program $p$, which takes as input TT-cores of the tensor $\tens{X}$ and outputs the value $f(\tens{X})$ in $F$ FLOP, which is polynomial with respect to the TT-ranks of the tensor $\tens{X}$ (i.e., the program $p$ belongs to the P complexity class). Then, the complexity of using Alg.~\ref{alg:tensor-riemannian-autodiff} for computing the TT-cores of the Riemannian gradient $\mathrm{P}_{\tens{X}} \nabla f$ is $\compl(F + d n \rank^3)$, where $n = \max_{k=1, \ldots, d} n_k$, $\rank = \max_{k=1, \ldots, d-1} r_k$.
\end{statement}

\begin{proof}

Let us estimate the complexity of each step of Alg.~\ref{alg:tensor-riemannian-autodiff}.

\textit{Step 1} consists in orthogonalizing the cores of the tensor $\tens{X}$ and can be done in $\compl(d n \rank^3)$ FLOP~\cite[end of Sec.~3]{oseledets2011ttMain}.

\textit{Steps 3 and 11} are running Alg.~\ref{alg:deltas-to-tangent-space} which consist of copying and rearranging some of the arrays which already exist in the memory. Therefore, it has linear complexity with respect to the sizes of the arrays, i.e., at most $\compl(d n r^2)$.

\textit{Step 4} computes the output of the program $p$ on TT-cores $\widehat{\mG}_1, \ldots, \widehat{\mG}_d$. Under the assumptions of the statement, the complexity $F$ of the function evaluation is polynomial with respect to the TT-rank $\rank$. 
Let $q$ be the degree of this polynomial. 
Since the TT-cores $\widehat{\mG}_1, \ldots, \widehat{\mG}_d$ define a TT-tensor with TT-ranks $2 \mathbf{r}$ --- twice larger when compared to the original TT-rank $\mathbf{r}$, the program $p$ will be executed on these TT-cores with the complexity $\compl(2^q F) = \compl(F)$. %
Thus, the complexity of evaluating the function $g$ at a given point is at most $\compl(F)$.

\textit{Step 5} uses classic automatic differentiation to compute the gradient of the function $g$ with respect to its arguments. Since the asymptotic complexity of the classical automatic differentiation equals the asymptotic complexity of computing the function at one point~\citep{baydin2015automatic}, this sub-step can also be done in $\compl(F)$ FLOP.

\textit{Steps 7, 8 and 10} consists in repeating the reshape operation $d-1$ times. The reshape operation can be done with constant complexity and in the worst case (when doing this operation in-place is not available) has the complexity equal to the size of arrays, i.e., $\compl(nr^2)$ per iteration.

\textit{Step 9} consists in evaluating the following expression $d-1$ times: $\tens{D}_k \coloneqq \tens{D}_k + \mU^L_{k}\left(\left(\mU^L_{k}\right)^\intercal\, \tens{D}_k\right)$. 
The multiplication $\left(\mU^L_{k}\right)^\intercal\, \tens{D}_k$ of a $r_k \times r_{k-1} n_k$ matrix times a $r_{k-1} n_k\times r_k$ matrix results into a $r_k \times r_k$ matrix and costs $\mathcal{O}(r_{k-1} r_k^2 n_k)$. The remaining operations are of the same or smaller asymptotic complexity.
Thus, updating all $\tens{D}_k$ can be done in $\compl(dn\rank^3)$ FLOP.

Summing the complexity across all steps yields the total complexity $\compl(F + d n \rank^3)$.
\end{proof}

For most functions used in practice, the asymptotic complexity $F$ of executing the function at one point exceeds $\compl(d n \rank^3)$ and the total complexity (as a function of $n$, $d$ and $\rank$) of the proposed algorithm equals to $\compl(F + d n \rank^3) = \compl(F)$. For example,  the functions listed at the end of Sec.~\ref{sec:riemannian-opt-briefer} (except for the recurrent neural network example) and their combinations such as
\[
	f(\mX) = \|\mathrm{P}_{\Omega}(\mX-\mA)\|^2+ \lambda \|\mX\|^2,
\]
are at least as expensive to evaluate as $\compl(d n \rank^3)$.

\subsection{Stop-gradient and a wider class of functionals}
\label{sec:stopgrad}

Suppose that we want to calculate projection to a tangent plane that cannot be easily associated with a Riemannian gradient of a functional.
As an example, in~\cite{kressner2016preconditioned}, to solve a linear system $\mathrm{A} \mX = \mathbf{F}$, a preconditioned version of the Riemannian gradient descent was considered: 
\begin{equation}\label{eq:precgrad}
	\mX_{k+1} = \mX_{k} - \tau_k \mathrm{P}_{\mX_k} \mathrm{B} \left(\mathrm{A} \mX_k - \mathbf{F}\right),
\end{equation}
where $\mathrm{B}$ is a preconditioner and $\tau_k\in\mathbb{R}$ is an iteration parameter.
If $\mathrm{B}$ is an identity operator and $\mathrm{A}$ is symmetric positive-definite, then the iteration~\eqref{eq:precgrad} is a Riemannian gradient descent associated with the function
\begin{equation}\label{eq:auxfunc}
f_{\mathrm{A}}(\mX) = \frac{1}{2}\left<\mathrm{A}\mX, \mX\right> -  \left<\mathbf{F}, \mX\right>.
\end{equation}
The problem is that to obtain $\mathrm{P}_{\mX_k} \mathrm{B} \left(\mathrm{A} \mX_k - \mathbf{F}\right)$, we cannot simply calculate the Riemannian gradient of \eqref{eq:auxfunc} with $\mathrm{B}\mathrm{A}$ instead of $\mathrm{A}$, and $\mathrm{B}\mathbf{F}$ instead of $\mathbf{F}$, since $\mathrm{B}\mathrm{A}$ is, in general, not symmetric even if both $\mathrm{A}$ and $\mathrm{B}$ are.
A similar problem arises for preconditioned eigensolvers.
To overcome it, we will use the notion of the \emph{stop-gradient operator} which is available in most automatic differentiation frameworks. 

The stop-gradient operator $c(\tens{X})$ is formally defined by the following two properties $c(\tens{X}) = \tens{X}$ and $\nabla c(\tens{X}) = \tens{O}$ --- zero tensor of the same size as $\tens{X}$.
It allows avoiding differentiating some parts of an expression when applying automatic differentiation. 
For example, for $x \in \mathbb{R}$ the derivative of $g(x) \equiv f(x c(x))$ is $g'(x) = f'(x^2)$ instead of $f'(x^2) 2x$.

We, thus, can (in the code) replace the function $f_{\mathrm{A}}$ with $h_{\mathrm{A},\mathrm{B}}$:
\[
	h_{\mathrm{A},\mathrm{B}}(\mX) = \left<\mathrm{B}\mathrm{A}\,c(\mX), \mX\right> -  \left<\mathrm{B}\mathbf{F}, \mX\right>.
\]
As a result, we obtain 
\begin{equation}\label{eq:precriem}
\mathrm{P}_{\mX} \nabla h_{\mathrm{A},\mathrm{B}}(\mX) = \mathrm{P}_{\mX} \mathrm{B}\left(\mathrm{A} \mX - \mathbf{F}\right),
\end{equation}
so we can simply apply the proposed AD approach to $h_{\mathrm{A},\mathrm{B}}(\mX)$.
Note that
\[
	\left<\mathrm{B}\mathrm{A}\,c(\mX), \mX\right> = \left<\mathrm{A}c(\mX), \mathrm{B}^\intercal \mX\right>
\]
and it can be implemented in $\compl(d n R_A R_B r^3 + d n^2 (R_A + R_B) R_A R_B r^2)$ FLOP.
Hence, using the proposed AD, we can calculate the Riemannian gradient~\eqref{eq:precriem} with the same asymptotic complexity.
If $\mathrm{B}$ is a sum of $\rho_{\mathrm{B}}$ rank-1 terms, for example, for a preconditioner based on exponential sums~\cite{khor-low-rank-kron-P1-2006,khor-prec-2009}, then the complexity can be additionally reduced to $\compl(d n  R_A \rho_{\mathrm{B}} r^3 + d n^2 \rho_{\mathrm{B}} R_A^2 r^2)$.

\section{Approximate Hessian-by-vector product} \label{sec:hess}
In this section, we show how to compute the product between the approximate Riemannian Hessian and a vector from the tangent space~\eqref{eq:ghess}. 

In the classical autodiff, there are two main ways of implementing Hessian-by-vector products given first-order autodiff implementation. The first approach consists in computing the gradient $\nabla f(\mathbf{x})$, then defining an auxiliary function $w\colon \mathbb{R}^{n} \rightarrow \mathbb{R}$, $w(\mathbf{x}) = \langle \nabla f(\mathbf{x}), ~\mathbf{z} \rangle$ and finally using first-order autodiff on the auxiliary function $\nabla^2 f(\mathbf{x})~ \mathbf{z} = \nabla w(\mathbf{x})$. The second approach consists in defining an auxiliary function $h\colon \mathbb{R} \rightarrow \mathbb{R}^n$, $h(t) = \nabla_{\mathbf{x}} f(\mathbf{x} + t \mathbf{z})$ by using first-order autodiff at the point $\mathbf{x} + t \mathbf{z}$, and then using forward mode autodiff\footnote{Using reverse mode autodiff would not be efficient in this case as the function $h$ has non-scalar output.} on the auxiliary function $h$ at the point $t=0$ to get the Hessian-by-vector product $\nabla^2 f(\mathbf{x})~ \mathbf{z} = h'(t)|_{t=0}$ (see e.g.~\cite{pearlmutter1994fast} for more details).

Both of these classical approaches can be generalized to the Riemannian case. Here we focus on the first approach, as the second approach requires forward mode autodiff which is not natively supported by the autodiff library we use for numerical experiments (TensorFlow). In the generalization of the first approach we additionally use the fact that when computing the auxiliary scalar-product function $w$, we are working with two vectors from the same tangent space. This allows us to compute their inner product more efficiently than in the general case.

Recall the definition of the approximate Riemannian Hessian by vector product
\begin{equation}
\label{eq:hessian-by-vector}
\mathrm{H}_\mX [\mZ] = \mathrm{P}_{{\tens{X}}} \nabla^2 f(\tens{X}) \, \tens{Z}, ~~ \mZ \in T_\mX \mathcal{M}.
\end{equation}
Note that the exact (non-approximate) Riemannian Hessian~\eqref{eq:hess} also includes the term for the derivative of the projection operator $\mathrm{P}_{{\tens{X}}}$ with respect to the tensor $\tens{X}$, which we ignore in~\eqref{eq:hessian-by-vector}.

Let us transform~\eqref{eq:hessian-by-vector} using the fact that $\tens{Z} \in T_\mX \mathcal{M}$, which implies $\tens{Z} = \mathrm{P}_{{c(\tens{X})}}\tens{Z}$. 
Note that we use the stop-gradient operator $c$ defined in Sec.~\ref{sec:stopgrad} to make sure that we are computing the approximate Riemannian Hessian, i.e., that we are not differentiating the projection operator $\mathrm{P}_{{c(\tens{X})}}$.
In this case,
\begin{equation*}
\mathrm{H}_\mX [\mZ] = \mathrm{P}_{{\tens{X}}} \nabla^2 f(\tens{X}) ~ \tens{Z} = \mathrm{P}_{{\tens{X}}} \frac{\partial }{\partial \tens{X}} \left \langle \nabla f,  \mathrm{P}_{{c(\tens{X})}}\tens{Z} \right \rangle
\end{equation*}
Using the symmetry of the orthogonal projection $\mathrm{P}_{c({\tens{X}})}$, we may write
\begin{equation*}
\mathrm{H}_\mX [\mZ] = \mathrm{P}_{{\tens{X}}} \frac{\partial }{\partial \tens{X}} \left \langle \nabla f,  \mathrm{P}_{{c(\tens{X})}} \tens{Z} \right\rangle = \mathrm{P}_{{\tens{X}}} \frac{\partial }{\partial \tens{X}} \left \langle \mathrm{P}_{{c(\tens{X})}} \nabla f, \tens{Z} \right \rangle.
\end{equation*}
Assume that we have access to the Riemannian gradient with the stop-gradient operator applied to the projection $\mathrm{P}_{{c(\tens{X})}} \nabla f$ (see below on how to obtain it). Then, we can compute %
\begin{equation}
\label{eq:riemannian-gradient-by-vector}
w(\tens{X}) = \left \langle \mathrm{P}_{{c(\tens{X})}} \nabla f, \tens{Z} \right \rangle
\end{equation}
and use the first-order Riemannian autodiff to find $\mathrm{H}_\mX [\mZ] = \mathrm{P}_{{\tens{X}}} \nabla w(\tens{X})$ -- the desired approximate Riemannian Hessian-by-vector product.

Note that~\eqref{eq:riemannian-gradient-by-vector} is a scalar product of two vectors belonging to the same tangent plane.
Let us consider this operation in more detail.
Suppose we are given two tensors $\mY,\mZ \in T_\mX \mathcal{M}$.
If $\mathcal{M}$ is a manifold of fixed-rank matrices, we can parametrize $\mY$ and $\mZ$ with the matrices $\dot{\mU}_\mY,\dot{\mV}_\mY$ and $\dot{\mU}_\mZ,\dot{\mV}_\mZ$ (see~\eqref{eq:matrix_tangent_space}). 
Hence,
\begin{equation}\label{eq:tangmatfast}
\left \langle \mY, \mZ \right \rangle 
= 
	\left<\dot{\mU}_\mY \mV^\intercal + \mU \dot{\mV}_\mY^\intercal, 
	\dot{\mU}_\mZ \mV^\intercal + \mU \dot{\mV}_\mZ^\intercal \right>
	=
	\left \langle \dot{\mU}_\mZ,  \dot{\mU}_\mY \right \rangle + \left \langle \dot{\mV}_\mZ,  \dot{\mV}_\mY \right \rangle .
\end{equation}

Similarly, if $\mathcal{M}$ is the manifold of fixed-rank TT tensors and $\mY,\mZ\in T_\mX \mathcal{M}$ are parametrized as in~\eqref{eq:rank-d-tangent-space} by $\{\dot{\mS}^{\mY}_k\}_{k=1}^d$ and $\{\dot{\mS}^{\mZ}_k\}_{k=1}^d$ respectively then
\begin{equation}\label{eq:tangtensfast}
\left \langle \tens{Y}, \tens{Z} \right \rangle = \sum_{k=1}^d \left \langle  \dot{\mS}^{\mY}_k,  \dot{\mS}^{\mZ}_k \right \rangle.
\end{equation}
Note that equations~\eqref{eq:tangmatfast} and \eqref{eq:tangtensfast} for matrices and tensors from the same tangent plane lead to faster computation of scalar products than for two general tensors of the same rank (see \cite[Sec. 4.4.4]{steinlechner2016riemannian}).

One might think that the first-order Riemannian autodiff described in Sec.~\ref{sec:low-rank-matrix-autodiff} and~\ref{sec:riemannian-grad-tensors} yields $\mathrm{P}_{{\tens{X}}} \nabla f$ instead of $\mathrm{P}_{{c(\tens{X})}} \nabla f$ and thus can not be utilized here. However, since first-order Riemannian autodiff works by differentiating at $\tens{X}$ the auxiliary function $g$ defined on a linear space $T_{\mX} \mathcal{M}_r$, a Riemannian gradient obtained this way lacks any information about the nonlinearity of the manifold. 
So, the method for computing the Riemannian gradient (Sec.~\ref{sec:low-rank-matrix-autodiff} and~\ref{sec:riemannian-grad-tensors}) actually yields $\mathrm{P}_{{c(\tens{X})}} \nabla f$. This nuance is irrelevant when computing the first-order Riemannian gradient because the two quantities coincide in value, but it becomes important when differentiating through this operation. Thus, we can reuse the proposed first-order Riemannian gradient to compute the product between the approximate Riemannian Hessian and a given vector with the method described above.

The algorithms to compute the multiplication of the approximate Riemannian Hessian by a vector are summarized in Alg.~\ref{alg:matrix-riemannian-hessian-by-vector} for the matrix case, and in Alg.~\ref{alg:tensor-riemannian-hessian-by-vector} for the tensor case.
Note that they require only a few additional operations compared to the algorithm for computing the Riemannian gradient.

\begin{algorithm}[t]
\caption{Computing the approximate Riemannian Hessian by vector product for low-rank matrices via AD.}\label{alg:matrix-riemannian-hessian-by-vector}
\begin{algorithmic}[1]
\Require $\mX = \mU\mS\mV^\intercal \in\mathbb{R}^{m \times n}$, matrices $\mdu^\mathbf{Z}, \mdv^\mathbf{Z}$ that define $\mathbf{Z} = \mdu^\mathbf{Z} \mV^\intercal + \mU ({\mdv^\mathbf{Z}})^\intercal \in T_\mX \mathcal{M}_r$ (see \eqref{eq:matrix_tangent_space}), $p(\mathbf{L}, \mathbf{R})$ -- implementation of evaluating $f$ at $\mathbf{L}\mathbf{R}^\intercal$ for any $\mathbf{L}\in\mathbb{R}^{m\times 2r}$ and $\mathbf{R}\in\mathbb{R}^{n\times 2r}$. 
\Ensure $\mdu, \mdv$ such that $\mathrm{H}_\mX [\mZ] = \mathrm{P}_\mX \nabla^2 f(\mX) \, \mZ = \mdu \mV^\intercal + \mU \mdv^\intercal$
\vspace{0.1cm}
\Function{g}{$\mA, \mB$}
  \State \Return $p([\mU \, \mA],\, [\mB\, \mV])$
\EndFunction
\vspace{0.1cm}
\Function{w}{$\widehat\mA, \widehat\mB$}
  \State  $\mdu := \left.\frac{\partial \textproc{g}}{\partial \mA}\right|_{(\mA,\mB) = (\widehat \mA, \widehat \mB)}$ using AD
  \State $\mdv := \left.\frac{\partial \textproc{g}}{\partial \mB}\right|_{(\mA,\mB) = (\widehat \mA, \widehat \mB)}$ using AD 
  \State $\mdv^\intercal := \mdv^\intercal - (\mdv^\intercal \mV) \mV^\intercal$
  \State $\Return \left \langle  \mdu^\mZ, \mdu \right \rangle + \left \langle  \mdv^\mZ, \mdv \right \rangle$
\EndFunction
\vspace{0.1cm}
\State $\mdu := \left.\frac{\partial \textproc{w}}{\partial \mA}\right|_{(\mA,\mB) = (\mU \mS,  \mathbf{O})}$ using AD\;
\State $\mdv := \left.\frac{\partial \textproc{w}}{\partial \mB}\right|_{(\mA,\mB) = (\mU \mS,  \mathbf{O})}$ using AD\;
 \State $\mdv^\intercal := \mdv^\intercal - (\mdv^\intercal \mV) \mV^\intercal$
\end{algorithmic}
\end{algorithm}

\begin{algorithm}[th!]
\caption{Computing the approximate Riemannian Hessian by vector product for low-rank tensors via AD.}\label{alg:tensor-riemannian-hessian-by-vector}
\begin{algorithmic}[1]
        \Require $\{\mG_k\}_{k=1}^d$ -- TT-cores of $\mX$, the delta terms $\dot{\mS}_1^{\tens{Z}}, \ldots, \dot{\mS}_d^{\tens{Z}}$ that define the projection (onto the tangent space) of the tensor $\tens{Z}$ which has to be multiplied by the approximate Riemannian Hessian, $p(\widehat\mG_1, \ldots, \widehat\mG_d)$ -- Python implementation of $f(\widehat\mX)$ for a point $\widehat\mX$ given by TT-cores $\widehat\mG_1, \ldots, \widehat\mG_d$.
        \Ensure The TT-cores $\{\tens{H}_k\}_{k=1}^d$ of the approximate Riemannian Hessian by vector product~\eqref{eq:hessian-by-vector} %
\Statex
\State For $\mX$, compute left- and right-orthogonal TT-cores $\{\mU_k\}_{k=1}^{d-1}$, $\{\mV_k\}_{k=2}^{d}$ respectively and $\{\mS_k\}_{k=1}^d$ as in~\eqref{eq:orthogonalized-tt}.
\Function{g}{$\mR_1, \ldots, \mR_d$}
  \State Run Alg.~\ref{alg:deltas-to-tangent-space} passing as input $\{\mG_k\}_{k=1}^d$ and $\{\mR_k\}_{k=1}^d$ and write the output TT-cores into~$\{\widehat{\mG}_k\}_{k=1}^d$\;
  \State \Return $p(\widehat{\mG}_1, \ldots, \widehat{\mG}_d)$
\EndFunction
\Function{w}{$\widehat{\mR}_1, \ldots, \widehat{\mR}_d$}
  \State Using AD compute $\dot{\mS}_k := \left.\frac{\partial g}{\partial \mR_k}\right|_{(\mR_1, \mR_2, \ldots, \mR_d) = (\widehat{\mR}_1, \ldots, \widehat{\mR}_d)}$ for $k = 1, \ldots, d$\;
  \For{$k \gets 1$ to $d-1$}
  \State $\tens{D}_k := \reshape(\dot{\mS}_k, (r_{k-1} n_k, r_k))$
  \State $\mU^L_{k} := \reshape(\mU_{k}, (r_{k-1} n_k, r_k))$
  \State $\tens{D}_k := \tens{D}_k + \mU^L_{k}\left(\left(\mU^L_{k}\right)^\intercal\, \tens{D}_k\right)$\; \Comment{See~\eqref{eq:tt-enforcing-gauge-conditions}}
  \State $\dot{\mS}_k := \reshape(\tens{D}_k, (r_{k-1}, n_k, r_k))$
  \EndFor
  \State \Return $\sum_{k=1}^d \left \langle \dot{\mS}_k, \dot{\mS}_k^{\tens{Z}} \right \rangle$
\EndFunction
\State Using AD compute $\dot{\mS}_k := \left.\frac{\partial w}{\partial \mR_k}\right|_{(\mR_1, \mR_2, \ldots, \mR_d) = (\mS_1, \mathbf{O}_2, \ldots, \mathbf{O}_d)}$ for $k = 1, \ldots, d$\;%
    \For{$k \gets 1$ to $d-1$}
  \State $\tens{D}_k := \reshape(\dot{\mS}_k, (r_{k-1} n_k, r_k))$
  \State $\mU^L_{k} := \reshape(\mU_{k}, (r_{k-1} n_k, r_k))$
  \State $\tens{D}_k := \tens{D}_k + \mU^L_{k}\left(\left(\mU^L_{k}\right)^\intercal\, \tens{D}_k\right)$\; \Comment{See~\eqref{eq:tt-enforcing-gauge-conditions}}
  \State $\dot{\mS}_k := \reshape(\tens{D}_k, (r_{k-1}, n_k, r_k))$
  \EndFor
\State Run Alg.~\ref{alg:deltas-to-tangent-space} passing as input $\{\mG_k\}_{k=1}^d$ and $\{\dot{\mS}_k\}_{k=1}^d$ and write the output TT-cores into $\{\tens{H}_k\}_{k=1}^d$\;
\end{algorithmic}
\end{algorithm}

Let us estimate the complexity of the proposed algorithm.
\begin{statement}
\label{thm:matrix-riemannian-hess-complexity}
Let $f: \mathbb{R}^{m \times n} \rightarrow \mathbb{R}$ be a smooth function defined by a program $p$, which takes as input SVD decomposition of a matrix $\mX = \mU\mS\mV^\intercal \in\mathbb{R}^{m\times n}$ and outputs the value $f(\mX)$ in $F = F(m,n,r)$ floating point operations (FLOP), which is polynomial with respect to the rank of the matrix $\tens{X}$ (i.e., the program $p$ belongs to the P complexity class). Then, the complexity of using Alg.~\ref{alg:matrix-riemannian-hessian-by-vector} for computing delta terms $\mdu$ and $\mdv$ which define the product of the approximate Riemannian Hessian by a given vector (for the manifold of fixed-rank matrices) $\mathrm{H}_\mX [\mZ] = \mathrm{P}_\mX \nabla^2 f(\mX) \, \mZ = \mdu \mV^\intercal + \mU \mdv^\intercal$ is $\compl(F + n \rank^2)$.
\end{statement}
\begin{proof}
The algorithm for computing the approximate Riemannian Hessian by vector product in the matrix case (Alg.~\ref{alg:matrix-riemannian-hessian-by-vector}) is  similar to the algorithm for computing the Riemannian gradient (Alg.~\ref{alg:matrix-riemannian-autodiff}): subfunctions $g$ are identical in both algorithms, steps 4--6 and 8--10 in Alg.~\ref{alg:matrix-riemannian-hessian-by-vector} are identical to steps 3--5 Alg.~\ref{alg:matrix-riemannian-autodiff} (so it at most doubles the work and does not affect the asymptotic complexity). The only new operation is computing the dot product between the tangent space elements (step 7) which takes $\compl(n r^2)$ arithmetic operations. Thus, computing the approximate Riemannian Hessian by vector product asymptotic complexity is still $\compl(F + n r^2)$.
\end{proof}
As is noted at the end of Sec.~\ref{sec:riemannian-grad-mat-complexity}, for most practical functions $f(\mX)$ the complexity $F$ of evaluating the function at a single point dominates the added complexity $\compl(n \rank^2)$ of the proposed algorithm, making the total complexity of the algorithm coincide with the complexity of evaluating the function: $\compl(F + n \rank^2) = \compl(F)$.

\begin{statement}
\label{thm:riemannian-hess-complexity}
Let $f: \mathbb{R}^{n_1 \times \ldots \times n_d} \rightarrow \mathbb{R}$ be a smooth function defined by a program $p$, which takes as input TT-cores of the tensor $\tens{X}$ and outputs the value $f(\tens{X})$ in $F$ FLOP, which is polynomial w.r.t. the TT-ranks of the tensor $\tens{X}$ (i.e., the program $p$ belongs to the P complexity class). Then, the complexity of Algorithm~\ref{alg:tensor-riemannian-hessian-by-vector} for computing the product of the approximate Riemannian Hessian by a given vector (for the manifold of tensors of fixed TT-rank) $\mathrm{P}_{{\tens{X}}} \nabla^2 f(\tens{X})~\tens{Z}$ is $\compl(F + d n \rank^3)$, where $n = \max_{k=1, \ldots, d} n_k$, $\rank = \max_{k=1, \ldots, d-1} r_k$.
\end{statement}

\begin{proof}%
Similarly to the first-order case, let us estimate the complexity of each step of Alg.~\ref{alg:tensor-riemannian-hessian-by-vector}.

\textit{Steps 1, 2, 3, 4} define the function $g(\mR_1, \ldots, \mR_d)$ that can be evaluated at a given point in $\compl(F + d n r^3)$ FLOP (equivalently to the first-order case, see proof of Statement~\ref{thm:riemannian-grad-complexity} for details).

\textit{Steps 6-11} use classic automatic differentiation to compute the gradient of the function $g$ with respect to its arguments and then project the resulting gradients onto the gauge conditions. These steps are equivalent to steps 5-10 of Alg.~\ref{alg:tensor-riemannian-autodiff} and can be done in $\compl(F + d n r^3)$ (again, see proof of Statement~\ref{thm:riemannian-grad-complexity} for details).

\textit{Step 12} computes the dot product between two elements of the same tangent space, which (as noted above) can be computed with the complexity that equals to the number of elements in the delta-terms, i.e., $\compl(n d \rank^2)$.
So the total complexity of evaluating the function $w(\widehat{\mR}_1, \ldots, \widehat{\mR}_d)$ at a point is $\compl(F + n d \rank^3)$ FLOP.

\textit{Step 13} uses classic automatic differentiation to compute the gradient of $w(\widehat{\mR}_1, \ldots, \widehat{\mR}_d)$ with respect to its arguments, which can be done in $\compl(F + d n r^3)$ FLOP.

\textit{Steps 14-19} are equivalent to steps 6--11 of Alg.~\ref{alg:tensor-riemannian-autodiff} and (as discussed in the proof of Statement~\ref{thm:riemannian-grad-complexity}) take at most $\compl(d n r^3)$ FLOP.

Combining the complexity from all the steps yields $\compl(F + d n \rank^3)$.
\end{proof}

Similarly to the matrix case, for most practical functions $f(\mX)$, the total complexity of the algorithm is $\compl(F + d n \rank^3) = \compl(F)$.

\section{Numerical experiments}
In this section, we compare three ways of computing Riemannian gradients and approximate Riemannian Hessian-by-vector products: `naive' -- by deriving the expression for the TT-format of the Euclidean gradient and then projecting the Euclidean gradient onto the tangent space\footnote{Note that in this process we never materialize the dense representation of any tensor and always work with TT-representations.}; `improved' -- similar to the `naive' approach, but with additional tricks to speed up the computations using optimized primitives\footnote{Examples of additional tricks used: implementing projection of matrix-by-vector multiplication $P_{\tens{x}} \mathrm{A} \tens{b}$ as a single operation, instead of a doing them one-by-one allows to speed things up; using the fact that projection is a linear operation and thus $P_{\tens{X}} \sum_i \tens{A}_i = \sum_i P_{\tens{X}} \tens{A}_i$.} implemented in~\cite{novikov2018t3f}; `AD' -- by using the proposed automatic differentiation method. All methods give the same answer (as verified by tests for all the functions described below), so we only consider speed and memory usage when comparing the methods.

We ran the experiments on a machine with 240 Gb of RAM and an NVIDIA V100 GPU which has 16 Gb of video memory available.
For each problem, we tried to choose a realistic problem size (specified separately for each particular function below) and ran all the experiments with three different tiers of TT-ranks: Small, Medium and Large. We choose the Large ranks for each problem to be the largest TT-rank that fits RAM of the machine we ran the experiments on (240 Gb), Medium to be the largest TT-rank that fits the GPU memory (16 Gb), and Small TT-ranks to be twice smaller than the Medium TT-ranks. See Table~\ref{tbl:experiment-ranks} for the TT-ranks for each function.

\subsection{Functions} \label{sec:exp-functionals}
Below we talk in detail about the five functions considered in numerical experiments.

\paragraph{Quadratic form}
The first function we consider is quadratic form $f(\mX) = \left<\mathrm{A}\mX, \mX\right>$ with symmetric $\mathrm{A}$, which is relevant for solving systems of linear equations.
 Its Euclidean gradient equals $2 \mathrm{A}\mX$, and the product of its Hessian by a given vector $\tens{Z}$ equals $2 \mathrm{A}\tens{Z}$. 
`Naive' method for computing the projection of the matrix-by-vector product (e.g., the product of the approximate Riemannian Hessian by a given vector $\mathrm{P}_{\tens{X}} 2\mathrm{A}\tens{Z}$) consist in first computing the matrix-by-vector product $\mathrm{A}\tens{Z}$ and then projecting the result. The combined complexity of the `naive' approach is $\compl(d n r_x r_z^2 R^2)$, where the TT-rank of the tensor $\tens{X}$ is $\mathbf{r}_x=(r_x,r_x,\dots,r_x)$, TT-rank of tensor $\tens{Z}$ is $\mathbf{r}_z=(r_z,r_z,\dots,r_z)$, and TT-rank of the operator $\mathrm{A}$ is $\mathbf{R}=(R,R,\dots,R)$. An `improved' version of this operation combines the matrix-by-vector multiplication and the projection onto the tangent space into a single step  $\mathrm{P}_{\tens{X}}  \mathrm{A}\tens{Z}$ and exploits the structure of arising operations to decrease complexity to $\compl(d n^2 r_x r_z R^2)$ (for details of implementation of this operation see section 4.1 of~\cite{rakhuba2019low}).

In the experiments below, we consider a 40-dimensional tensor $\tens{X} \in \mathbb{R}^{20 \times \ldots \times 20}$ and represent the operator $\mathrm{A}$ by a TT-matrix of size $20^{40} \times 20^{40}$. We use TT-ranks $r_{\mathrm{A}} = 10$, $r_{\tens{X}} = 10$, $r_{\tens{Z}} = 20$ for the Small TT-rank experiment and $r_{\mathrm{A}} = 20$, $r_{\tens{X}} = 20$, $r_{\tens{Z}} = 40$ for the Medium and Large TT-rank experiments.

\paragraph{Quadratic form with a Gram matrix}
The second function is quadratic form $f(\mX) = \left<\mathrm{A}^\intercal\mathrm{A}\mX, \mX\right>$ (operator factored into the product of two TT-matrices $\mathrm{A}^\intercal\mathrm{A}$ arised, e.g., in~\cite{bachmayr2020stability}). 
The Euclidean gradient equals to $2\mathrm{A}^\intercal\mathrm{A}\mX$ and the product of its Hessian by a given vector $\tens{Z}$ equals to $2\mathrm{A}^\intercal\mathrm{A}\tens{Z}$. 
We use the same trick to optimize the projection of the product of two matrices by a vector as in the quadratic form case. Note that it takes significant effort to derive and implement the `improved' version here.

In the experiments below, we consider a 10-dimensional tensor $\tens{X} \in \mathbb{R}^{20 \times \ldots \times 20}$ and represent the operator $\mathrm{A}$ by a TT-matrix of size $20^{10} \times 20^{10}$. We use TT-ranks $r_{\mathrm{A}} = 10$, $r_{\tens{X}} = 5$, $r_{\tens{Z}} = 10$ for the Small TT-rank experiment, $r_{\mathrm{A}} = 20$, $r_{\tens{X}} = 10$, $r_{\tens{Z}} = 20$ for the Medium TT-rank experiments, and $r_{\mathrm{A}} = 20$, $r_{\tens{X}} = 20$, $r_{\tens{Z}} = 40$ for the Large TT-rank experiments.

\paragraph{Rayleigh quotient}
The Rayleigh quotient $f(\mX) = {\left<\mathrm{A} [\mX], \mX\right>}/{\left<\mX, \mX\right>}$ with symmetric $\mathrm{A}$ is relevant for solving eigenvalue problems. 
The Euclidean gradient is $\frac{2}{\langle \mX, \mX \rangle} (\mathrm{A} [\mX] - {f(\mX)}\mX)$, and the product of its Hessian by a given vector $\tens{Z}$ is 
\begin{equation*}
\begin{aligned}
\nabla^2 f(\tens{X}) ~\tens{Z} =& \frac{2}{\langle \mX, \mX \rangle} \mathrm{A} \tens{Z} - 2 \frac{f(\mX)}{\langle \mX, \mX \rangle} \tens{Z} - 4 \frac{\langle \mathrm{A} \mX, \tens{Z} \rangle}{\langle \mX, \mX \rangle^2}\mX \\
&- 4 \frac{\langle \mX, \tens{Z} \rangle}{\langle \mX, \mX \rangle^2}\mathrm{A} \mX + 8f(\mX)\frac{\langle \mX, \tens{Z} \rangle}{\langle \mX, \mX \rangle^2} \mX
\end{aligned}
\end{equation*}

The `improved' version of the Riemannian gradient and approximate-Riemannian-Hessian-by-vector product is computed by representing the projection of a sum of terms as a sum of projections and using the optimized projection of matrix-by-vector multiplication where appropriate, e.g., for the Riemannian gradient we get
\[
\mathrm{P}_\mX\, \nabla f = \frac{2}{\langle \mX, \mX \rangle} \mathrm{P}_\mX\, \mathrm{A} \mX - \frac{2f(\mX)}{\langle \mX, \mX \rangle} \mX,
\]
where we use the fact that $\mathrm{P}_\mX\, \mX = \mX$.

In the experiments below we consider a 40-dimensional tensor $\tens{X} \in \mathbb{R}^{20 \times \ldots \times 20}$ and represent the operator $\mathrm{A}$ by a TT-matrix of size $20^{40} \times 20^{40}$. We use TT-ranks $r_{\mathrm{A}} = 10$, $r_{\tens{X}} = 10$, $r_{\tens{Z}} = 20$ for the Small TT-rank experiment and $r_{\mathrm{A}} = 20$, $r_{\tens{X}} = 20$, $r_{\tens{Z}} = 40$ for the Medium and Large TT-rank experiments.

\begin{table}[t]
\centering
\begin{tabular}{lcc cc cc}
\toprule
          Function &  \multicolumn{2}{c}{Small} & \multicolumn{2}{c}{Medium} & \multicolumn{2}{c}{Large} \\
          \cmidrule(lr){2-3}\cmidrule(lr){4-5}\cmidrule(r){6-7}
           &  tensor $\tens{X}$ & operator $\mathrm{A}$ & tensor $\tens{X}$ & operator $\mathrm{A}$ & tensor $\tens{X}$ & operator $\mathrm{A}$ \\
\midrule
              $\left< \mathrm{A} \mX, \mX \right>$ &       10 &       10 &          20 &          20 &         20 &         20 \\
             $\left< \mathrm{A}^\intercal \mathrm{A} \mX , \mX \right>$ &       10 &        5 &          20 &          10 &         20 &         20 \\
 RayleighQuotient &       10 &       10 &          20 &          20 &         20 &         20 \\
       completion &        5 &        - &          10 &           - &         20 &          - \\
      ExpMachines &        5 &        - &          10 &           - &         20 &          - \\
\bottomrule
\end{tabular}
\caption{Ranks of tensors, matrices and vectors involved in different tiers of experiments. See Sec.~\ref{sec:exp-functionals} for more details. \label{tbl:experiment-ranks}}
\end{table}

\paragraph{Completion problem}
The following function is used when solving low-rank matrix and tensor completion problems: $f(\mX) = \|\mathrm{P}_\Omega (\mX - \ma)\|^2$ where $\mathrm{P}_\Omega$ denotes projection on the index set $\Omega$ such that
\[
  \mathrm{P}_\Omega \mX = 
  \begin{cases}
    X_{i_1\dotsi_d} \quad &(i_1,\dots,i_d) \in \Omega,\\
    0 & \text{otherwise}.
  \end{cases}
\]
Its Euclidean gradient is $2 \mathrm{P}_\Omega (\mX - \ma)$, and the product of its Euclidean Hessian by a given vector $\tens{Z}$ is $\mathrm{P}_\Omega \tens{Z}$. For the `improved' implementation, we represent the projection of a tensor on the index set as the sum of its non-zero entries $\mathrm{P}_\Omega \mX = \sum_{(i_1,\dots,i_d) \in \Omega} X_{i_1 \dots i_d} \tens{E}^{i_1\dots i_d}$ where by~$\tens{E}^{i_1 \dots i_d}$ we denote the tensor with value~1 in the position $(i_1,\dots,i_d)$ and zero everywhere else. Tensor $\tens{E}^{i_1 \dots i_d}$ has TT-rank~1. 
Then, we use the fact that the projection of a sum of terms is the sum of projections and, thus, instead of projecting the tensor $\mathrm{P}_\Omega \mX$ which has high TT-rank, we project TT-rank-1 tensors $\tens{E}$, which leads to a significant speed-up.

\begin{table}[t]
\centering
\begin{subtable}[h]{\textwidth}
\centering
\begin{tabular}{lcc cc cc}
\toprule
          Function &  \multicolumn{2}{c}{Naive} & \multicolumn{2}{c}{Improved} & \multicolumn{2}{c}{AD} \\
          \cmidrule(lr){2-3}\cmidrule(lr){4-5}\cmidrule(r){6-7}
           &  (s) & (Gb) &  (s) & (Gb) &  (s) & (Gb) \\
\midrule
                      $\left< \mathrm{A} \mX , \mX \right>$ &        2.4 &          6.4 &        3.3 &          2.8 &      \textbf{1} &  \textbf{0.62} \\
 $\left< \mathrm{A}^\intercal \mathrm{A} \mX , \mX \right>$ &          2 &           10 &          - &            - &   \textbf{0.54} &   \textbf{0.3} \\
                                           RayleighQuotient &        3.1 &          6.9 &        3.4 &          2.8 &    \textbf{1.1} &  \textbf{0.62} \\
                                                 completion &          - &            - &        3.4 &           13 &   \textbf{0.98} &   \textbf{6.2} \\
                                                ExpMachines &       0.18 &        0.082 &       0.12 &        0.042 &  \textbf{0.078} &  \textbf{0.03} \\
\bottomrule
\end{tabular}
   \caption{Comparison of computing Riemannian gradient by three methods on CPU for Medium TT-ranks. \label{tbl:riemannian-autodiff-grad-cpu-medium}}
\end{subtable}
~\\[0.3cm]
\begin{subtable}[h]{\textwidth}
\centering
\begin{tabular}{lcc cc cc}
\toprule
          Function &  \multicolumn{2}{c}{Naive} & \multicolumn{2}{c}{Improved} & \multicolumn{2}{c}{AD} \\
          \cmidrule(lr){2-3}\cmidrule(lr){4-5}\cmidrule(r){6-7}
           &  (s) & (Gb) &  (s) & (Gb) &  (s) & (Gb) \\
\midrule
                      $\left< \mathrm{A} \mX , \mX \right>$ &       0.17 &           11 &  \textbf{0.057} &  \textbf{0.56} &           0.085 &            0.62 \\
 $\left< \mathrm{A}^\intercal \mathrm{A} \mX , \mX \right>$ &          - &            - &               - &              - &  \textbf{0.032} &   \textbf{0.28} \\
                                           RayleighQuotient &       0.22 &           12 &   \textbf{0.07} &  \textbf{0.57} &             0.1 &            0.63 \\
                                                 completion &          - &            - &               - &              - &   \textbf{0.25} &    \textbf{5.6} \\
                                                ExpMachines &      0.034 &         0.11 &  \textbf{0.027} &           0.04 &  \textbf{0.027} &  \textbf{0.017} \\
\bottomrule
\end{tabular}
   \caption{Comparison of computing Riemannian gradient by three methods on GPU for Medium TT-ranks. \label{tbl:riemannian-autodiff-grad-gpu-medium}}
\end{subtable}
\caption{Comparison of three methods for Medium TT-rank setting (see Table~\ref{tbl:experiment-ranks}) for computing  the Riemannian gradient of various functions in terms of execution time and memory used on CPU and GPU. A dash means that the respective method ran out of memory. \label{tbl:riemannian-autodiff-grad-medium}}
\end{table}

In the experiments below we consider a 10-dimensional tensor $\tens{X} \in \mathbb{R}^{20 \times \ldots \times 20}$ and the index set $\Omega$ consisting of $10 d n r_{\mathrm{X}}^2$ elements (i.e., the TT-rank of $\mathrm{P}_\Omega \mX$ equals to $10 d n r_{\mathrm{X}}^2$ and does not fit to memory for the `naive' implementation even in the Small TT-rank case). The values of the target tensor~$\tens{A}$ at the randomly chosen $10 d n r_{\mathrm{X}}^2$ elements are sampled from the standard normal distribution. For the Small TT-rank experiment, we use $r_{\mathrm{X}} = 5$, $r_{\tens{Z}} = 10$ and $10 d n r_{\mathrm{X}}^2 = 50,000$ observed elements; for Medium TT-rank experiment, we use $r_{\mathrm{X}} = 10$, $r_{\tens{Z}} = 20$ and $10 d n r_{\mathrm{X}}^2 = 200,000$ observed elements; for Large TT-rank experiment, we use $r_{\mathrm{X}} = 20$, $r_{\tens{Z}} = 40$ and $10 d n r_{\mathrm{X}}^2 = 800,000$ observed elements;. 

\paragraph{Exponential machines}
For a machine learning related function, we used the empirical risk of the exponential machines model (see~\cite{novikov18exponential} for details and justification): 
\[f(\mX) = \sum_{i=1}^N h(\langle \mX, \tens{W}^{(i)} \rangle, y^{(i)}),\] 
where $h(x, y)$ is the loss function $h(x, y) = \log(1 + \exp(-yx))$\footnote{This loss adapted from~\cite{novikov18exponential} is equivalent to the cross-entropy loss when the label $y$ takes values from $ \{-1, 1\}$ instead of the more common $ \{0, 1\}$.}, tensors $\tens{W}^{(i)}$ have TT-rank 1, and $y^{(i)}$ are binary numbers.

As argued in~\cite{novikov18exponential}, this model corresponds to a type of recurrent neural network, so we refer to this example as a neural network loss in the rest of the paper.

The gradient of this function is
\[
\nabla f = -\sum_{i=1}^N \frac{\exp(-y^{(i)}\langle \mX, \tens{W}^{(i)} \rangle)}{1 + \exp(-y^{(i)}\langle \mX, \tens{W}^{(i)} \rangle)} \tens{W}^{(i)}
\] and the product of the Hessian of this function by a given vector is 
\[
\nabla^2 f(\tens{X})\tens{Z} = \sum_{i=1}^N \frac{\exp(-y^{(i)}\langle \mX, \tens{W}^{(i)} \rangle)}{(1 + \exp(-y^{(i)}\langle \mX, \tens{W}^{(i)} \rangle)^2}\langle \tens{Z}, \tens{W}^{(i)} \rangle \tens{W}^{(i)}.
\]
Again, by using linearity of the projection, to implement the `improved' version we can independently compute the cheap projections $\mathrm{P}_\mX\, \tens{W}^{(i)}$ and sum them up.

In the experiments below we consider a 10-dimensional tensor $\tens{X} \in \mathbb{R}^{500 \times \ldots \times 500}$ (which corresponds to a machine learning problem with 10 categorical features each of which can take $500$ different values) and number of objects (in the minibatch) $N=32$. We use TT-ranks $r_{\mathrm{X}} = 5$, $r_{\tens{Z}} = 10$ for the Small TT-rank experiment,  $r_{\mathrm{X}} = 10$, $r_{\tens{Z}} = 20$ for the Medium TT-rank experiments, and $r_{\mathrm{X}} = 20$, $r_{\tens{Z}} = 40$ for the Large TT-rank experiments.

\subsection{Results}
We used the T3F library~\citep{novikov2018t3f} for implementing all three algorithms for the five functions described above. The T3F library provides the primitives used above such as the optimized projection of a matrix by vector product and also supports GPU execution (thanks to the underlying use of TensorFlow library~\citep{tensorflow2015-whitepaper}). We implemented the Riemannian automatic differentiation functionality as a part of the T3F library as well.

\begin{table}[t]
\centering
\begin{subtable}[h]{\textwidth}
\centering
\begin{tabular}{lcc cc cc}
\toprule
          Function &  \multicolumn{2}{c}{Naive} & \multicolumn{2}{c}{Improved} & \multicolumn{2}{c}{AD} \\
          \cmidrule(lr){2-3}\cmidrule(lr){4-5}\cmidrule(r){6-7}
           &  (s) & (Gb) &  (s) & (Gb) &  (s) & (Gb) \\
\midrule
                      $\left< \mathrm{A} \mX , \mX \right>$ &        6.9 &           28 &            3.6 &             3.2 &  \textbf{2.2} &   \textbf{1.1} \\
 $\left< \mathrm{A}^\intercal \mathrm{A} \mX , \mX \right>$ &          5 &           36 &              - &               - &  \textbf{1.1} &  \textbf{0.49} \\
                                           RayleighQuotient &         18 &           56 &            4.9 &               4 &  \textbf{2.4} &   \textbf{1.1} \\
                                                 completion &          - &            - &            3.5 &              22 &  \textbf{2.6} &    \textbf{12} \\
                                                ExpMachines &       0.21 &        0.075 &  \textbf{0.12} &  \textbf{0.053} &          0.13 &          0.079 \\
\bottomrule
\end{tabular}
   \caption{Comparison of computing the approximate Riemannian Hessian by vector product by three methods on CPU for Medium TT-ranks. \label{tbl:riemannian-autodiff-hessian-by-vector-cpu-medium}}
\end{subtable}
~\\[0.3cm]
\begin{subtable}[h]{\textwidth}
\centering
\begin{tabular}{lcc cc cc}
\toprule
          Function &  \multicolumn{2}{c}{Naive} & \multicolumn{2}{c}{Improved} & \multicolumn{2}{c}{AD} \\
          \cmidrule(lr){2-3}\cmidrule(lr){4-5}\cmidrule(r){6-7}
           &  (s) & (Gb) &  (s) & (Gb) &  (s) & (Gb) \\
\midrule
                      $\left< \mathrm{A} \mX , \mX \right>$ &          - &            - &  \textbf{0.067} &  \textbf{0.66} &           0.16 &              1 \\
 $\left< \mathrm{A}^\intercal \mathrm{A} \mX , \mX \right>$ &          - &            - &               - &              - &  \textbf{0.06} &  \textbf{0.54} \\
                                           RayleighQuotient &          - &            - &   \textbf{0.14} &  \textbf{0.73} &           0.19 &            1.1 \\
                                                 completion &          - &            - &               - &              - &  \textbf{0.64} &    \textbf{11} \\
                                                ExpMachines &      0.036 &         0.11 &  \textbf{0.028} &  \textbf{0.04} &           0.03 &          0.043 \\
\bottomrule
\end{tabular}
   \caption{Comparison of computing the approximate Riemannian Hessian by vector product by three methods on GPU for Medium TT-ranks. \label{tbl:riemannian-autodiff-hessian-by-vector-gpu-medium}}
\end{subtable}
\caption{Comparison of three methods for Medium TT-rank setting (see Table~\ref{tbl:experiment-ranks}) for computing the approximate Riemannian Hessian by vector product of various functions in terms of execution time and memory used on CPU and GPU. A dash means that the respective method ran out of memory. \label{tbl:riemannian-autodiff-hessian-by-vector-medium}}
\end{table}

Results of the main numerical experiments are presented in Table~\ref{tbl:riemannian-autodiff-hessian-by-vector-medium}, plus additional results on Small and Large ranks are presented in Appendix. The proposed automatic differentiation method outperformed both the `naive' and the `improved' implementations for computing the Riemannian gradient on CPU both in terms of runtime and memory usage (Tables~\ref{tbl:riemannian-autodiff-grad-cpu-small},\ref{tbl:riemannian-autodiff-grad-cpu-medium} and~\ref{tbl:riemannian-autodiff-grad-cpu-large}).

Note that sometimes the `improved' implementation runs slower than the `naive' implementation. After profiling our implementation of the methods we believe that this happens because the `improved' implementation operates with tensors of larger dimensionality (e.g., when the `naive' version operates with a tensor of size $1024 \times 1024 \times 1024$, the `improved' implementation may operate with the same tensor, but reshaped to $32 \times 32 \times 32 \times 32 \times 32 \times 32$ for better flexibility), which causes an additional overhead when permuting dimensions. This overhead is typically neglectable when executing on GPU, because GPUs have the required flops to compute all the necessary permute operations in parallel.

In some cases, the proposed automatic differentiation method is outperformed in terms of the runtime by the `improved' implementation. This happens due to the (constant) overhead that arises when performing automatic differentiation. For example, when computing the approximate Riemannian Hessian-by-vector product of the quadratic form, the `improved' implementation can directly compute the desired quantity $\mathrm{P}_{\tens{X}} \mathrm{A}\tens{Z}$, while the automatic differentiation is forced to first compute the function $\left<\mathrm{A}\mX, \mX\right>$, then perform the classic automatic differentiation twice, each time doubling the computational graph, making the computational graph four times larger than the original one.

However, we believe that despite some overheads compared to the `improved' implementation which appears in individual cases, the proposed method is still valuable since it significantly simplifies the implementation of Riemannian optimization algorithms while getting reasonable (and in many cases superior) performance.

\section{Conclusion}
In this paper, we propose a way of exactly 	computing the Riemannian gradient and the approximate Riemannian Hessian-by-vector product of a function for low-rank matrices and tensors in time proportional to the time it takes to compute the value of the function at one point.
In experiments, the proposed approach in many cases shows superior performance compared to both considered baselines in terms of memory and time, while being significantly easier to use.
The code of the proposed algorithms is published online in the open-source library T3F.

\section*{Acknowledgements}
This work was supported by the Ministry of Science and Higher Education of the Russian Federation (Grant \#075-15-2020-801).

\bibliography{autodiff,tensor,our}

\newpage
\appendix
\section{Additional experimental results}

Here we provide additional experimental results. If in the main text, only the Medium TT-rank experiments were provided, here we also provide results on input tensors of Small and Large TT-ranks (see Sec.~\ref{sec:exp-functionals} for a detailed explanation of the setup and of the TT-ranks chosen for all experiments).

\begin{table}[h!]
\centering
\begin{subtable}[h]{\textwidth}
\centering
\begin{tabular}{lcc cc cc}
\toprule
          Function &  \multicolumn{2}{c}{Naive} & \multicolumn{2}{c}{Improved} & \multicolumn{2}{c}{AD} \\
          \cmidrule(lr){2-3}\cmidrule(lr){4-5}\cmidrule(r){6-7}
           &  (s) & (Gb) &  (s) & (Gb) &  (s) & (Gb) \\
\midrule
                      $\left< \mathrm{A} \mX , \mX \right>$ &       0.32 &         0.42 &       0.43 &         0.26 &    \textbf{0.2} &     \textbf{0.1} \\
 $\left< \mathrm{A}^\intercal \mathrm{A} \mX , \mX \right>$ &      0.098 &         0.15 &          - &            - &  \textbf{0.085} &   \textbf{0.033} \\
                                           RayleighQuotient &       0.52 &         0.48 &       0.49 &         0.22 &   \textbf{0.23} &    \textbf{0.11} \\
                                                 completion &          - &            - &       0.46 &         0.88 &  \textbf{0.074} &    \textbf{0.41} \\
                                                ExpMachines &       0.14 &        0.061 &      0.034 &        0.013 &  \textbf{0.026} &  \textbf{0.0097} \\
\bottomrule
\end{tabular}
   \caption{Comparison of computing Riemannian gradient by three methods on CPU for Small TT-ranks. \label{tbl:riemannian-autodiff-grad-cpu-small}}
\end{subtable}
~\\[0.3cm]
\begin{subtable}[h]{\textwidth}
\centering
\begin{tabular}{lcc cc cc}
\toprule
          Function &  \multicolumn{2}{c}{Naive} & \multicolumn{2}{c}{Improved} & \multicolumn{2}{c}{AD} \\
\cmidrule(lr){2-3}\cmidrule(lr){4-5}\cmidrule(r){6-7}
           &  (s) & (Gb) &  (s) & (Gb) &  (s) & (Gb) \\
\midrule
                      $\left< \mathrm{A} \mX , \mX \right>$ &   \textbf{0.029} &         0.69 &  \textbf{0.029} &         0.14 &           0.036 &     \textbf{0.1} \\
 $\left< \mathrm{A}^\intercal \mathrm{A} \mX , \mX \right>$ &  \textbf{0.0083} &         0.23 &               - &            - &            0.01 &   \textbf{0.031} \\
                                           RayleighQuotient &   \textbf{0.038} &         0.76 &           0.039 &         0.14 &           0.042 &    \textbf{0.11} \\
                                                 completion &                - &            - &           0.092 &          1.2 &  \textbf{0.062} &    \textbf{0.37} \\
                                                ExpMachines &            0.031 &          0.1 &  \textbf{0.011} &        0.013 &  \textbf{0.011} &  \textbf{0.0048} \\
\bottomrule
\end{tabular}
   \caption{Comparison of computing Riemannian gradient by three methods on GPU for Small TT-ranks. \label{tbl:riemannian-autodiff-grad-gpu-small}}
\end{subtable}
\caption{Comparison of three methods for Small TT-rank setting (see Table~\ref{tbl:experiment-ranks}) for computing the Riemannian gradient of various functions in terms of execution time and memory used on CPU and GPU. A dash means that the respective method ran out of memory. \label{tbl:riemannian-autodiff-grad-small}}
\end{table}

\begin{table}[h!]
\centering
\begin{subtable}[h]{\textwidth}
\centering
\begin{tabular}{lcc cc cc}
\toprule
          Function &  \multicolumn{2}{c}{Naive} & \multicolumn{2}{c}{Improved} & \multicolumn{2}{c}{AD} \\
          \cmidrule(lr){2-3}\cmidrule(lr){4-5}\cmidrule(r){6-7}
           &  (s) & (Gb) &  (s) & (Gb) &  (s) & (Gb) \\
\midrule
                      $\left< \mathrm{A} \mX , \mX \right>$ &        2.4 &          6.4 &        3.3 &          2.8 &     \textbf{1} &  \textbf{0.62} \\
 $\left< \mathrm{A}^\intercal \mathrm{A} \mX , \mX \right>$ &         17 &          162 &          - &            - &   \textbf{1.9} &   \textbf{1.1} \\
                                           RayleighQuotient &        3.1 &          6.9 &        3.4 &          2.8 &   \textbf{1.1} &  \textbf{0.62} \\
                                                 completion &          - &            - &          - &            - &    \textbf{13} &    \textbf{98} \\
                                                ExpMachines &       0.35 &         0.12 &          1 &         0.15 &  \textbf{0.33} &  \textbf{0.11} \\
\bottomrule
\end{tabular}
   \caption{Comparison of computing Riemannian gradient by three methods on CPU for Large TT-ranks. \label{tbl:riemannian-autodiff-grad-cpu-large}}
\end{subtable}
~\\[0.3cm]
\begin{subtable}[h]{\textwidth}
\centering
\begin{tabular}{lcc cc cc}
\toprule
          Function &  \multicolumn{2}{c}{Naive} & \multicolumn{2}{c}{Improved} & \multicolumn{2}{c}{AD} \\
          \cmidrule(lr){2-3}\cmidrule(lr){4-5}\cmidrule(r){6-7}
           &  (s) & (Gb) &  (s) & (Gb) &  (s) & (Gb) \\
\midrule
                      $\left< \mathrm{A} \mX , \mX \right>$ &           0.17 &           11 &  \textbf{0.057} &  \textbf{0.56} &          0.085 &            0.62 \\
 $\left< \mathrm{A}^\intercal \mathrm{A} \mX , \mX \right>$ &              - &            - &               - &              - &  \textbf{0.15} &   \textbf{0.94} \\
                                           RayleighQuotient &           0.22 &           12 &   \textbf{0.07} &  \textbf{0.57} &            0.1 &            0.63 \\
                                                 completion &              - &            - &               - &              - &              - &               - \\
                                                ExpMachines &  \textbf{0.15} &         0.12 &   \textbf{0.15} &           0.14 &  \textbf{0.15} &  \textbf{0.067} \\
\bottomrule
\end{tabular}
   \caption{Comparison of computing Riemannian gradient by three methods on GPU for Large TT-ranks. \label{tbl:riemannian-autodiff-grad-gpu-large}}
\end{subtable}
\caption{Comparison of three methods for Large TT-rank setting (see Table~\ref{tbl:experiment-ranks}) for computing the Riemannian gradient of various functions in terms of execution time and memory used on CPU and GPU. A dash means that the respective method ran out of memory. \label{tbl:riemannian-autodiff-grad-large}}
\end{table}

\begin{table}[h!]
\centering
\begin{subtable}[h]{\textwidth}
\centering
\begin{tabular}{lcc cc cc}
\toprule
          Function &  \multicolumn{2}{c}{Naive} & \multicolumn{2}{c}{Improved} & \multicolumn{2}{c}{AD} \\
          \cmidrule(lr){2-3}\cmidrule(lr){4-5}\cmidrule(r){6-7}
           &  (s) & (Gb) &  (s) & (Gb) &  (s) & (Gb) \\
\midrule
                      $\left< \mathrm{A} \mX , \mX \right>$ &       0.71 &          1.7 &   \textbf{0.43} &         0.33 &           0.48 &   \textbf{0.17} \\
 $\left< \mathrm{A}^\intercal \mathrm{A} \mX , \mX \right>$ &       0.32 &         0.56 &               - &            - &  \textbf{0.14} &  \textbf{0.058} \\
                                           RayleighQuotient &        1.8 &          4.2 &             0.8 &         0.55 &  \textbf{0.52} &   \textbf{0.17} \\
                                                 completion &          - &            - &            0.47 &         0.91 &  \textbf{0.19} &    \textbf{0.9} \\
                                                ExpMachines &       0.22 &        0.056 &  \textbf{0.047} &        0.032 &          0.057 &   \textbf{0.03} \\
\bottomrule
\end{tabular}
   \caption{Comparison of computing the approximate Riemannian Hessian by vector product by three methods on CPU for Small TT-ranks. \label{tbl:riemannian-autodiff-hessian-by-vector-cpu-small}}
\end{subtable}
~\\[0.3cm]
\begin{subtable}[h]{\textwidth}
\centering
\begin{tabular}{lcc cc cc}
\toprule
          Function &  \multicolumn{2}{c}{Naive} & \multicolumn{2}{c}{Improved} & \multicolumn{2}{c}{AD} \\
          \cmidrule(lr){2-3}\cmidrule(lr){4-5}\cmidrule(r){6-7}
           &  (s) & (Gb) &  (s) & (Gb) &  (s) & (Gb) \\
\midrule
                      $\left< \mathrm{A} \mX , \mX \right>$ &           0.053 &          2.7 &  \textbf{0.035} &         0.21 &           0.066 &   \textbf{0.16} \\
 $\left< \mathrm{A}^\intercal \mathrm{A} \mX , \mX \right>$ &  \textbf{0.018} &         0.67 &               - &            - &           0.019 &  \textbf{0.058} \\
                                           RayleighQuotient &            0.16 &            7 &           0.093 &         0.58 &  \textbf{0.079} &    \textbf{0.2} \\
                                                 completion &               - &            - &    \textbf{0.1} &          1.8 &            0.15 &   \textbf{0.72} \\
                                                ExpMachines &           0.031 &          0.1 &  \textbf{0.011} &        0.013 &           0.014 &  \textbf{0.012} \\
\bottomrule
\end{tabular}
   \caption{Comparison of computing the approximate Riemannian Hessian by vector product by three methods on GPU for Small TT-ranks. \label{tbl:riemannian-autodiff-hessian-by-vector-gpu-small}}
\end{subtable}
\caption{Comparison of three methods for Small TT-rank setting (see Table~\ref{tbl:experiment-ranks}) for computing the approximate Riemannian Hessian by vector product of various functions in terms of execution time and memory used on CPU and GPU. A dash means that the respective method ran out of memory. \label{tbl:riemannian-autodiff-hessian-by-vector-small}}
\end{table}

\begin{table}[h!]
\centering
\begin{subtable}[h]{\textwidth}
\centering
\begin{tabular}{lcc cc cc}
\toprule
          Function &  \multicolumn{2}{c}{Naive} & \multicolumn{2}{c}{Improved} & \multicolumn{2}{c}{AD} \\
          \cmidrule(lr){2-3}\cmidrule(lr){4-5}\cmidrule(r){6-7}
           &  (s) & (Gb) &  (s) & (Gb) &  (s) & (Gb) \\
\midrule
                      $\left< \mathrm{A} \mX , \mX \right>$ &            6.9 &             28 &        3.6 &          3.2 &  \textbf{2.2} &  \textbf{1.1} \\
 $\left< \mathrm{A}^\intercal \mathrm{A} \mX , \mX \right>$ &              - &              - &          - &            - &  \textbf{4.2} &    \textbf{2} \\
                                           RayleighQuotient &             18 &             56 &        4.9 &            4 &  \textbf{2.4} &  \textbf{1.1} \\
                                                 completion &              - &              - &          - &            - &             - &             - \\
                                                ExpMachines &  \textbf{0.38} &  \textbf{0.11} &          1 &         0.14 &          0.52 &          0.23 \\
\bottomrule
\end{tabular}
   \caption{Comparison of computing the approximate Riemannian Hessian by vector product by three methods on CPU for Large TT-ranks. \label{tbl:riemannian-autodiff-hessian-by-vector-cpu-large}}
\end{subtable}
~\\[0.3cm]
\begin{subtable}[h]{\textwidth}
\centering
\begin{tabular}{lcc cc cc}
\toprule
          Function &  \multicolumn{2}{c}{Naive} & \multicolumn{2}{c}{Improved} & \multicolumn{2}{c}{AD} \\
          \cmidrule(lr){2-3}\cmidrule(lr){4-5}\cmidrule(r){6-7}
           &  (s) & (Gb) &  (s) & (Gb) &  (s) & (Gb) \\
\midrule
                      $\left< \mathrm{A} \mX , \mX \right>$ &          - &              - &  \textbf{0.067} &  \textbf{0.66} &           0.16 &           1 \\
 $\left< \mathrm{A}^\intercal \mathrm{A} \mX , \mX \right>$ &          - &              - &               - &              - &  \textbf{0.31} &  \textbf{2} \\
                                           RayleighQuotient &          - &              - &   \textbf{0.14} &  \textbf{0.73} &           0.19 &         1.1 \\
                                                 completion &          - &              - &               - &              - &              - &           - \\
                                                ExpMachines &       0.16 &  \textbf{0.12} &   \textbf{0.15} &           0.14 &  \textbf{0.15} &        0.17 \\
\bottomrule
\end{tabular}
   \caption{Comparison of computing the approximate Riemannian Hessian by vector product by three methods on GPU for Large TT-ranks. \label{tbl:riemannian-autodiff-hessian-by-vector-gpu-large}}
\end{subtable}
\caption{Comparison of three methods for Large TT-rank setting (see Table~\ref{tbl:experiment-ranks}) for computing the approximate Riemannian Hessian by vector product of various functions in terms of execution time and memory used on CPU and GPU. A dash means that the respective method ran out of memory. \label{tbl:riemannian-autodiff-hessian-by-vector-large}}
\end{table}

\end{document}